\documentclass[11pt]{amsart}
\usepackage{amsmath,amssymb,enumerate,etoolbox,mathrsfs,xcolor,url,hyperref}
\usepackage[normalem]{ulem}

\usepackage[top=30mm,right=30mm,bottom=30mm,left=20mm]{geometry}

\newtheorem{theorem}{Theorem}[section]
\newtheorem{thm}[theorem]{Theorem}
\newtheorem{conjecture}{Conjecture}

\newtheorem{theorema}{Theorem}

\newtheorem{prop}[theorem]{Proposition}
\newtheorem*{proposition*}{Proposition}

\newtheorem{lem}[theorem]{Lemma}
\newtheorem{cor}[theorem]{Corollary}

\newcommand{\F}{\mathbb{F}}

\newcommand{\N}{\mathbb{N}}

\newcommand{\Z}{\mathbb{Z}}

\newcommand{\CB}{\mathbf{C}}

\newcommand{\SL}{{\rm SL}}
\newcommand{\GL}{{\rm GL}}
\newcommand{\SU}{{\rm SU}}
\newcommand{\GU}{{\rm GU}}

\newcommand{\Irr}{{\rm Irr}}

\newcommand{\Hom}{{\rm Hom}}

\newcommand{\Spec}{{\rm Spec}\,}

\newcommand{\Cl}{{\rm Cl}}

\newcommand{\cG}{\mathcal{G}}
\newcommand{\cL}{\mathcal{L}}
\newcommand{\cH}{\mathcal{H}}
\newcommand{\cT}{\mathcal{T}}
\newcommand{\cO}{\mathcal{O}}
\newcommand{\cR}{\mathcal{R}}
\newcommand{\cS}{\mathcal{S}}

\newcommand{\cM}{\mathcal{M}}

\newcommand{\SSS}{\mathsf{S}}
\newcommand{\eps}{\varepsilon}

\newcommand{\cl}{\mathfrak{l}}

\newcommand{\tcl}{\cl^*}
\newcommand{\tw}[1]{{}^#1\!}
\newcommand{\SR}{\tw* R}

\let\rC\CB
\let\FF\F
\let\QQ\Q
\let\ZZ\Z
\DeclareMathOperator{\Ind}{Ind}
\DeclareMathOperator{\Res}{Res}
\newcommand{\CC}{\mathbb{C}}
\newcommand{\cB}{\mathcal{B}}
\newcommand{\scC}{\mathscr{C}}
\newcommand{\rN}{\mathbf{N}}
\newcommand{\cf}{\mathrm{cf}}
\newcommand{\abs}[1]{\lvert #1 \rvert}
\renewcommand{\setminus}{\smallsetminus}
\renewcommand{\emptyset}{\varnothing}

\numberwithin{equation}{section}

\begin{document}

\title{Representation growth of Fuchsian groups and modular forms}

\author{Michael Larsen}
\email{mjlarsen@indiana.edu}
\address{Department of Mathematics\\
    Indiana University \\
    Bloomington, IN 47405\\
    U.S.A.}
    
\author{Jay Taylor}
\email{jay.taylor@manchester.ac.uk}
\address{
Department of Mathematics\\
The University of Manchester\\
Oxford Road\\
Manchester, M13 9PL\\
U.K.}    

\author{Pham Huu Tiep}
\email{pht19@math.rutgers.edu}
\address{Department of Mathematics\\
    Rutgers University \\
    Piscataway, NJ 08854\\
    U.S.A.}
\thanks{The first author was partially supported by the NSF 
grants DMS-2001349 and DMS-2401098. The third author gratefully acknowledges the support of the NSF (grant
DMS-2200850) and the Joshua Barlaz Chair in Mathematics.}

\dedicatory{To the memory of Gary Seitz}

\begin{abstract}
Let $\Gamma$ be a cocompact, oriented Fuchsian group which is not on an explicit finite
list of possible exceptions and $q$ a sufficiently large prime power not divisible by the order of any non-trivial torsion element of $\Gamma$.  Then
$|\Hom(\Gamma,\GL_n(q))|\sim c_{q,n} q^{(1-\chi(\Gamma))n^2}$, where $c_{q,n}$ is periodic in $n$.  As a function of $q$, $c_{q,n}$ can be expressed as a Puiseux series in $1/q$ whose coefficients are periodic in $n$ and $q$.  Moreover, this series is essentially the $q$-expansion of a meromorphic modular form of half-integral weight.
\end{abstract}

\maketitle

\tableofcontents

\section{Introduction}

Let $\Gamma$ be a cocompact and oriented Fuchsian group (which, in what follows, we shall call simply a Fuchsian group).  
Concretely, this means that $\Gamma$ has a presentation
$$\langle x_1,y_1,\ldots,x_g,y_g,z_1,\ldots,z_r \vert z_1^{a_1},\,\ldots, z_r^{a_r},\,[x_1,y_1]\cdots[x_g,y_g]z_1\cdots z_r\rangle,$$
where $a_1,a_2,\ldots,a_r$ is a fixed (possibly empty) non-decreasing sequence of integers $a_i\ge 2$
such that the Euler characteristic 
$$\chi(\Gamma) := 2-2g- \sum_{i=1}^r (1-\frac 1{a_i})$$
is negative.  Let $\F_q$ be a finite field.
In this paper, we investigate the asymptotic growth in $n$ of the number of homomorphisms from $\Gamma$ to the group $\GL_n(q)$, which we denote $G_n$ when the value of $q$ is understood.

There are two complementary points of view.  On the one hand we can fix $n$ and consider the homomorphism scheme $\Hom(\Gamma,\GL_n) := \Hom(\Gamma,\GL_{n,\Z})$, which is defined over $\Z$.
For fixed characteristic $p>0$, we can think of the fiber of $\Hom(\Gamma,\GL_n)$
over $\Spec\F_p$ as the variety $\Hom(\Gamma,\GL_{n,\F_p})$ of homomorphisms from $\Gamma$ to $\GL_n$ over $\F_p$.  (Note that in this paper, a \emph{variety} will be just an affine scheme of finite type over a field;
in particular, it need not be either irreducible or reduced.)
Applying the Lang-Weil theorem to this variety as $q$ ranges over powers of $p$, we see that
the number of homomorphisms $\rho\colon \Gamma\to \GL_n(q)$ determines its \emph{dimension}, by which we mean the maximum dimension of any of its irreducible components.

On the other hand, for fixed $q$, we can partition homomorphisms $\rho$ according to the $r$-tuple of $G_n$-conjugacy classes 
$$(C_1,\ldots,C_r) = (\rho(z_1)^{G_n},\ldots,\rho(z_r)^{G_n}).$$
Each $C_i$ must consist of elements of order dividing $a_i$.  For given $(C_1,\ldots,C_r)$ satisfying this divisibility condition, the number of homomorphisms $\Gamma\to G_n$ 
with $\rho(z_i)\in C_i$ for all $i$ is given by a theorem of Hurwitz \cite[Proposition 3.2]{LiSh}:
\begin{equation}
\label{Hurwitz}
|G|^{2g-1}|C_1|\cdots |C_r|\sum_{\chi\in \Irr(G_n)} \frac{\chi(C_1)\cdots \chi(C_r)}{\chi(1)^{2g+r-2}}.
\end{equation}
Summing \eqref{Hurwitz} over all possible $r$-tuples $(C_1,\ldots,C_r)$ we obtain a
formula which can potentially be used for fixed $q$ to understand the asymptotic behavior of $|\Hom(\Gamma,G_n)|$ as $n\to \infty$.  

These two ways of counting $\Hom(\Gamma,\GL_n(q))$ are in some sense complementary.
For instance, just as we can use character methods to determine the dimension of $\Hom(\Gamma,\GL_n)$, we can use the dimension of $\Hom(\Gamma,\GL_n)$ to get an upper bound on
$\Hom(\Gamma,\GL_n(q))$ for all large $q$. For the large $n$ limit,  the first method of counting seems to be the more appropriate, and an analysis of \eqref{Hurwitz} has led us to the following conjecture:

\begin{conjecture}
\label{Form}
Let $A$ denote the least common multiple of $a_1,\ldots,a_r$, which we take to be $1$ if $r=0$.
Let $q$ be a prime power relatively prime to $A$.
\begin{enumerate}[\rm(a)]
\item There exists a $2A$-periodic sequence $c_{q,1},c_{q,2},\ldots$ of positive numbers such that 
$$|\Hom(\Gamma,\GL_n(q))|\sim c_{q,n} q^{(1-\chi(\Gamma))n^2}.$$
\item There exist a $2$-dimensional array $e_{\Gamma,q,n}$ of rational numbers and a $2$-dimensional array  $f_{\Gamma,q,n}$ of half-integral weight meromorphic modular forms, periodic in both $q$ and $n$, 
such that
$$c_{q,n} = (q-1)q^{e_{\Gamma,q,n}} f_{\Gamma,q,n}\bigl(\frac{i\log q}{2\pi}\bigr).$$
Moreover, $f_{\Gamma,q,n}$ is holomorphic on the upper half plane and has integer Fourier coefficients at $i\infty$.
\end{enumerate}
\end{conjecture}

The problem of estimating the number of representations of a given Fuchsian group over a finite field
seems to have been first considered by Liebeck and Shalev \cite{LiSh1}.  There are a number of 
significant differences in emphasis between that paper and ours; $\Gamma$ need not be oriented in their paper, and the target of homomorphisms from $\Gamma$ could be a quasisimple group $G(q)$ instead of $\GL_n(q)$. 
They were primarily interested in the ``geometric'' direction, that is, $n$ fixed and $q\to \infty$..
A key limitation of their paper is that their method requires $g\ge 2$.

Under this hypothesis, they showed that the contribution in \eqref{Hurwitz} from non-linear characters is negligible, which reduces the problem of estimating 
$|\Hom(\Gamma,\GL_n(q))|$ to that of estimating the numbers $j_{q,n}(a_i)$ of elements $x\in \GL_n(q)$ satisfying $x^{a_i}=1$. They  gave an asymptotic formula for $j_{q,n}(a)$ when $n$ is fixed and $q\to \infty$, using work of Lawther \cite{Law}.  In the case $r=0$ (the surface group case) they proved that $|\Hom(\Gamma,\GL_n(q))|$ is asymptotic to $(q-1)|\GL_n(q)|^{2g-1}$.  If $\eta(z)$ is the Dedekind function,
then
$$\frac{|\GL_n(q)|}{q^{1/24}\eta(\frac{i\log q}{2\pi})} = q^{n^2}\prod_{i=n+1}^\infty (1-q^{-i})^{-1}\sim q^{n^2}.$$
Setting $e_{\Gamma,q,n} = \frac{2g-1}{24}$ and $f_{\Gamma,q,n} = \eta^{2g-1}$ for all $q,n$, we deduce Conjecture~\ref{Form} for surface groups.  In general, their
analysis depends crucially on the fact that for $g\ge 2$, the trivial upper bound on $|\chi(C_i)|$ is good enough to allow us to ignore non-linear characters of $G_n$.  This is certainly not the case
when $g=1$, let alone when $g=0$.
However, the new character bounds developed by Bezrukavnikov, Liebeck, Shalev, and Tiep \cite{BLST} give us hope of making progress even for $g=0$.

Using these bounds, Liebeck, Shalev, and Tiep proved \cite{LST} that for every $\Gamma$, if $q\equiv 1\pmod A$, then
$$|\Hom(\Gamma,\GL_n(q))| \le f(n) q^{(1-\chi(\Gamma))n^2+1}$$
where $f(n)$ does not depend on $q$.  This gives immediately an upper bound for the dimension of the representation variety $\Hom(\Gamma,\GL_{n,K})$ where $K$ is any field in which $A\neq 0$,
provided that $n$ is sufficiently large: 
$$\dim \Hom(\Gamma,\GL_{n,K}) \le (1-\chi(\Gamma))n^2+1.$$  
They proved also a lower bound on dimension:
$$\dim \Hom(\Gamma,\GL_{n,K})\ge (1-\chi(\Gamma))n^2 - \sum_i a_i.$$

In this paper, using ideas from \cite{BLST,TT}, we prove a new exponential character bound Theorem~\ref{appr-bound},
which applies to all semisimple elements and
which plays an essential role in the proof of the main theorems of this paper.  The reason we can prove Conjecture~\ref{Form} only when $q$
is sufficiently large is that the exponent in our bound only approaches its optimal value as $q\to\infty$.

Proposition~\ref{Genus 0} gives a list, consisting of thirty-one triangle groups and one quadrilateral group, 
where even for large $q$, our bounds are not strong enough to prove the conjecture.
%We do not know whether to expect that this is just a limitation of our method or reflects a real difference in asymptotic representation growth for these groups.
When $\Gamma$ is not on this list and $q$ is sufficiently large and relatively prime to the $a_i$, $|\Hom(\Gamma,G_n)|$
behaves as predicted.

\begin{theorema}
\label{Main}
Let $\Gamma$ be a Fuchsian group which is not on the finite list of groups excluded by Proposition \ref{Genus 0}. 
Then Conjecture~\ref{Form} holds for $\Gamma$ for all sufficiently large prime powers $q$ prime to $A$.
\end{theorema}

In particular, the theorem holds for all Fuchsian groups $\Gamma$ with Euler characteristic less than $-\frac 16$.

We deduce Theorem~\ref{Main} from an analogue of \cite[Theorem~1.2 (i)]{LiSh1}.
Let $J_{q,n}(a_1,\ldots,a_r)$ denote the cardinality of the set
\begin{equation}
\label{J set}
 \biggl\{(t_1,\ldots,t_r)\in \GL_n(q)\mid t_i^{a_i}=1\ \forall i,\ \prod_i \det(t_i)=1\biggr\}.
 \end{equation}

\begin{theorema}
\label{Roots}
If $\Gamma$ is not on the excluded list of Proposition \ref{Genus 0}, and $q$ is prime to $A$ and sufficiently large, then,
$$|\Hom(\Gamma,\GL_n(q))| = (1+o(1))(q-1) J_{q,n}(a_1,\ldots,a_r) |\GL_n(q)|^{2g-1},$$
where the term $o(1)$ does not depend on $q$.
\end{theorema}

Theorem~\ref{Main} allows us to compute the exact dimension of $\Hom(\Gamma,\GL_n)$ when $n$ is sufficiently large.  Given $\Gamma$ and $n$, we define $\sigma_{\Gamma,n}$ to be either $1$ or $-1$
according to the rule that it is $-1$ if and only if $a_i\in 2\Z$  implies $\frac n{a_i}\in\Z$, and 
$$\sum_{\{i\mid a_i\in 2\Z\}} \frac n{a_i} \in 1+2\Z.$$
Let $\{x\}$ denote the fractional part of $x$.

We have the following result:

\begin{theorema}
\label{Dimension}
Given $\Gamma$ which is not on the excluded list of Proposition \ref{Genus 0}, if $n$ is sufficiently large, and $K$ is any field of characteristic $p\ge 0$, if $p\nmid a_1$ for all $i$, we have
$$\dim \Hom(\Gamma,\GL_{n,K}) =\sigma_{\Gamma,n} + (1-\chi(\Gamma))n^2 - \sum_{i=1}^r a_i \Bigl\{\frac n{a_i}\Bigr\}\Bigl\{-\frac n{a_i}\Bigr\}.$$
In particular, we have
$$\dim \Hom(\Gamma,\GL_{n,K}) \ge -\frac12 + (1-\chi(\Gamma))n^2 - \sum_{i=1}^r \frac{a_i}4.$$
\end{theorema}

Unfortunately, the letter $q$ has a standard meaning both for finite fields and for modular forms.  We use it only in the former sense, but we evaluate modular forms $f$ at $\frac{i\log q}{2\pi}$, which amounts to plugging $1/q$ into the $q$-expansion for $f$.  

%A \emph{variety} in this paper need not be irreducible, and its \emph{dimension} will mean the maximum dimension of any of its irreducible components.

\section{An asymptotic character bound}

The goal of this section is to prove an asymptotic version of the character bounds in \cite[Theorem 1.1]{BLST} and 
\cite[Theorem 1.9]{TT} when $G$ is a finite group of Lie-type $A$. We will 
achieve this by combining the approach of \cite{BLST} with the {\it character level} approach developed in \cite{GLT} to bound $|\chi(g)|$.

To this end, let us recall the approach of \cite{BLST}. Throughout this section, $q$ is a prime power, $\cG = \cG(\overline{\F_q})$ is the group of $\overline{\F_q}$-points of a connected reductive $\F_q$-group scheme. We assume $G = \cG(\F_q) = \cG^F$ is the finite group of $\F_q$-points, where $F : \cG \to \cG$ is the Frobenius endomorphism.

The main case of interest to us will be when the underlying group scheme is $\GL_n^{\epsilon}$, where $\epsilon \in \{+,-\}$ and we set $\GL_n^+ := \GL_n$, the general linear group, and $\GL_n^- := \GU_n$, the general unitary group. In this setting,
%%%%
\begin{equation*}
\cG = \cG_n := \GL_n(\overline{\F_q})
\end{equation*}
%%%%
is the general linear group of dimension $n>0$ and $F$ is either $F_q$ or $\sigma F_q$, where $F_q : \cG \to \cG$ is the standard Frobenius endomorphism and $\sigma : \cG \to \cG$ is the inverse transpose automorphism.

Suppose $L = \cL^F$, where $\cL < \cG$ is a proper $F$-stable Levi subgroup of $\cG$. If $g \in \cG^F$ is an element such that $\rC_{\cG}(g) \leqslant \cL$ then by \cite[Lemma 13.3]{TT} we have
%%%%
\begin{equation}\label{eq-d20}
\chi(g)=\SR_{\cL}^{\cG}(\chi)(g) = \sum_{\eta \in \Irr(L)} \langle \eta,\SR_{\cL}^{\cG}(\chi)\rangle \eta(g),
\end{equation}
%%%%
where $\SR_{\cL}^{\cG}$ denotes Deligne--Lusztig restriction. We also write $R_{\cL}^{\cG}$ for Deligne--Lusztig induction.

Following \cite[Theorem 1.1]{BLST}, we define the constant $\alpha(L)$ to be the maximum over non-trivial unipotent elements $u\in L$ of
%%%%
\begin{equation*}
\frac{\dim u^{\cL}}{\dim u^{\cG}};
\end{equation*}
%%%%
if $L$ contains no such elements we take $\alpha(L)=0$. From the proof of \cite[Theorem~1.1]{BLST}, see also \cite[\S2]{TT}, we get
%%%%
\begin{equation}\label{eq-d21}
\abs{\eta(g)} \leqslant \eta(1) \leqslant B_1\left(\frac{q+1}{q-1}\right)^{D/2}\chi(1)^{\alpha(L_1)},
\end{equation}
%%%%
for any $\eta \in \Irr(L)$ with $\langle \eta,\SR_{\cL}^{\cG}(\chi)\rangle \neq 0$, where $B_1 > 0$ is a constant that depends on $\cG^F$. Furthermore, $D=\dim v^\cG$, where $v^\cG = \cO^*_\chi$ is the {\it wave front set} of $\chi$, defined by work of Kawanaka \cite{K}, Lusztig \cite{Lus}, and Taylor \cite{Ta}.

If $f_{\eta} \in \QQ[X]$ is the degree polynomial of $\eta$, so that $\eta(1) = f_{\eta}(q)$, then the constant $B_1$ is chosen such that $B_1f_{\eta} \in \ZZ[X]$. When the underlying group scheme is $\GL_n^{\epsilon}$ we have $\cL^F$ is a direct product of groups $\GL_{n_i}^{\epsilon_i}(q)$. Therefore, in this case, the constant $B_1$ can be taken to be $1$ because the degree polynomial of any irreducible character of $\cL^F$ is already contained in $\ZZ[X]$.

Recall that $\SR_{\cL}^{\cG}(\chi)$ is a virtual character, unless $\cL$ is split. To bound $\abs{\chi(g)}$ it suffices, by the triangle inequality, \eqref{eq-d20}, and \eqref{eq-d21}, to bound
%%%%
\begin{equation}\label{eq-d22}
\sum_{\eta \in \Irr(L)} \abs{\langle \eta,\SR_{\cL}^{\cG}(\chi)\rangle} \leqslant 
\sum_{\eta \in \Irr(L)} \langle \eta,\SR_{\cL}^{\cG}(\chi)\rangle^2
=
\langle \SR_{\cL}^{\cG}(\chi),\SR_{\cL}^{\cG}(\chi)\rangle,
\end{equation}
%%%%
where $\langle-,-\rangle$ is the usual inner product on class functions.

We know from \cite[Proposition 2.2]{BLST} and its proof, as well as the arguments in \cite[\S13]{TT}, that \eqref{eq-d22} is always bounded above by
%%%%
\begin{equation}\label{fn1}
(n!)^2.
\end{equation}
%%%%
However, we can do significantly better if $n$ is large compared to both $q$ and the {\it true level}
%%%%
\begin{equation*}
\tcl(\chi)=j
\end{equation*}
%%%%
of $\chi$, as defined in \cite[Definition 1(i)]{GLT}. When $G = \GL_n(q)$ then $j$ is the smallest integer for which $\chi$ is a constituent of $\tau^j$, where $\tau(g)$ is the number of fixed points of $g$ acting on the natural module $V = \FF_q^n$ of $G$.

In the next subsections we give upper bounds for \eqref{eq-d22} that incorporate the true level of $\chi$.

\subsection{Elements with split centralizer in $\GL_n(q)$}
In this subsection we consider the group scheme $\GL_n$ so that
$$G = G_n := \GL_n(q).$$ 
Fix a proper split Levi subgroup $\cL$, and let
\begin{equation}\label{levi-l}
  L= \cL^F=\GL_{m_1}(q) \times \GL_{m_2}(q) \times \ldots \times \GL_{m_t}(q)\subset G,
\end{equation}  
where $m_i \in \Z_{\geq 1}$ and $\sum^t_{i=1}m_i = n$. In this case $\SR_{\cL}^{\cG}$ is just Harish-Chandra restriction. With $1 \leq j < n/2$ fixed, consider a split Levi subgroup $\cM$ and set
$$M= \cM^F \cong \GL_j(q) \times \GL_{n-j}(q)\subset G.$$
By \cite[Theorem 3.9(i)]{GLT}, $\tcl(\chi) = j$ implies that 
$\chi$ is an irreducible constituent of the Harish-Chandra induction
$$R^{\cG}_{\cM}\bigl(\alpha \boxtimes 1_{G_{n-j}}\bigr)$$
for a unique irreducible character $\alpha$ of $G_j$. Conjugating $\cM$ by a suitable element $g \in G$, we may assume that
$L$ and $M$ are block-diagonal subgroups in the same basis $(e_1,e_2, \ldots,e_n)$ of $V$. 

To bound \eqref{eq-d22},  it therefore suffices to bound %the squared inner product of 
$$\bigl\langle\SR^{\cG}_{\cL} R^{\cG}_{\cM}\bigl(\alpha \boxtimes 1_{G_{n-j}}\bigr),\SR^{\cG}_{\cL} R^{\cG}_{\cM}\bigl(\alpha \boxtimes 1_{G_{n-j}}\bigr)\bigr\rangle.$$	
By the Mackey formula for Harish-Chandra restriction and induction \cite[Theorem 1.14]{DF}, 
\begin{equation}\label{hc1}
  \SR^{\cG}_{\cL} R^{\cG}_{\cM}\bigl(\alpha \boxtimes 1_{G_{n-j}}\bigr) = \sum_{x \in L \backslash \cS(\cL,\cM)/M}R^{\cL}_{\cL \cap \tw x\cM}\SR^{\tw x\cM}_{\cL \cap \tw x \cM}
    \bigl( (\alpha \boxtimes 1_{G_{n-j}})^x\bigr),
\end{equation}    
where $\cS(\cL,\cM)$ is the set of elements $y \in G$ such that $\cL \cap \tw y\cM$ contains a maximal torus of $\cG$, and the summation runs through
the $(L,M)$ double cosets of this set.

For our pair of split Levi subgroups $(\cL,\cM)$, there is an explicit description of $L \backslash \cS(\cL,\cM)/M$, as described in 
\cite[2.2c]{BDK}. Embed the symmetric group $\SSS_n$ in $G_n$ via permutation matrices, and consider the Young subgroups 
$$\SSS_\lambda = \SSS_{m_1} \times \SSS_{m_2} \times \ldots \times \SSS_{m_t},~~\SSS_\mu = \SSS_j \times \SSS_{n-j}$$
of the embedded $\SSS_n$. Then in \eqref{hc1} we can just choose $x$ as representatives of the set 
$\SSS_\lambda \backslash \SSS_n/\SSS_\mu$, one for each double coset. The set of double cosets $\SSS_\lambda \backslash \SSS_n/\SSS_\mu$
is in bijection with $\SSS_\lambda$-orbits on the set of $\SSS_n/\SSS_\mu$, which may be identified with the set of $j$-subsets of 
$\{1,2, \ldots,n\}$. Hence each such double coset can be labeled uniquely by a $t$-tuple
\begin{equation}\label{hc-k}
  \kappa=(k_1, k_2, \ldots,k_t),~~ 0 \leq k_i \leq m_i,~\sum^t_{i=1}k_i=j.
\end{equation}  
Correspondingly, we can choose $x=x_\kappa$ to be the element of $G$ that sends the first $j$ basis vectors $e_1, \ldots,e_j$ of $V$ 
to 
$$e_1, \ldots,e_{k_1},e_{m_1+1},e_{m_1+2}, \ldots,e_{m_1+k_2}, \ldots,e_{m_1+\ldots+m_{t-1}+1}, \ldots,e_{m_1+m_2 + \ldots +m_{t-1}+k_t}$$  
in the increasing order of the subscripts, and sends the last $n-j$ basis vectors $e_{j+1}, \ldots,e_n$ to the remaining $n-j$ basis vectors, 
again in the increasing order of the subscripts. We will say that $x_\kappa(e_i) = e_{x_\kappa(i)}$, $1 \leq i \leq n$.

For the reader's convenience, let us give a justification for this statement
in the case $q \geq 3$. Suppose $y \in G$ is such that $\cL \cap \tw y\cM$ contains a maximal torus $\cT$ of $\cG$. 
Then $\cT$ is a maximal torus of $(\cL \cap \tw y\cM)^\circ$ which is $F$-stable and connected. By the Lang--Steinberg theorem,
conjugating $\cT$ suitably, we may assume that it is $F$-stable. Then 
$$T:= \cT^F \cong C_{q^{a_1}-1} \times C_{q^{a_2}-1} \times \ldots \times C_{q^{a_s}-1}$$
for some integers $a_1, a_2, \ldots,a_s \geq 1$. Since $q \geq 3$, all cyclic direct factors in this decomposition are nontrivial, and hence 
$V =\F_q^n$ is a direct sum of $s$ simple $\F_qT$-modules $W_1, \ldots,W_s$ of dimension $a_1, a_2, \ldots,a_s$, which are pairwise non-isomorphic (indeed, they 
have pairwise distinct kernels). On the other hand, the $\F_qL$-module $V$ decomposes as 
the sum $\oplus^t_{i=1}V_i$ of $\F_qL$-modules, where 
$$V_1:= \langle e_1, \ldots,e_{m_1} \rangle_{\F_q},~V_2:= \langle e_{m_1+1}, \ldots,e_{m_1+m_2} \rangle_{\F_q}, \ldots,
    V_t:= \langle e_{m_1+ \ldots +m_{t-1}+1}, \ldots,e_{n} \rangle_{\F_q}.$$
As $T \leq L$, each $V_i$ is a direct sum of some of these $W_l$, $1 \leq l \leq s$. 
Similarly, since $V = \F_q^j \oplus \F_q^{n-j}$ as an $\F_q\tw yM$-module and 
$T \leq \tw yM$, each of $\F_q^j$ and $\F_q^{n-j}$ is a direct sum of some of these $W_l$. Using the left multiplcation by $L$ and right multiplication by
$M$ if needed, we may assume that $\F_q^j$ is spanned by 
$$e_1, \ldots,e_{k_1},e_{m_1+1},e_{m_1+2}, \ldots,e_{m_1+k_2}, \ldots,e_{m_1+\ldots+m_{t-1}+1}, \ldots,e_{m_1+m_2 + \ldots +m_{t-1}+k_t}$$
($k_i$ first vectors in the indicated basis of $V_i$ for each $1 \leq i \leq t$), and  $\F_q^{n-j}$ is spanned by the remaining $n-j$ basis vectors.

It is well known (and can be proved by an easy induction on $t \geq 1$) that the total number $N$ of $t$-tuples $\kappa$ in \eqref{hc-k} is 
\begin{equation}\label{hc-n}
  N= \binom{j+t-1}{j} = t \cdot \frac{t+1}{2} \cdots \frac{t+j-1}{j} \leq t^j \leq n^j
\end{equation}
since $t \leq n$. For each such $\kappa$, $x=x_\kappa$ sends $e_i$ to $e_{x(i)}$, and we can write
$$\tw xM = xMx^{-1} = \GL\bigl(\langle e_{x(1)}, \ldots,e_{x(j)}\rangle_{\F_q}\bigr) \times  \GL\bigl(\langle e_{x(j+1)}, \ldots,e_{x(n)}\rangle_{\F_q}\bigr)
\cong \GL_j(q) \times \GL_{n-j}(q)$$
Now, $L \cap \tw xM$ fixes each of the subspaces 
$$\langle e_1, \ldots,e_{m_1}\rangle_{\F_q} \cap \langle e_{x(1)}, \ldots,e_{x(j)}\rangle_{\F_q} = \langle e_{1}, \ldots,e_{k_1}\rangle_{\F_q}$$
and 
$$\langle e_1, \ldots,e_{m_1}\rangle_{\F_q} \cap \langle e_{x(j+1)}, \ldots,e_{x(n)}\rangle_{\F_q} = \langle e_{k_1+1}, \ldots,e_{m_1}\rangle_{\F_q}$$
of $V_1$, and similarly for $V_i$ with $2 \leq i \leq t$. It follows that
$$L \cap \tw xM = \prod^t_{i=1}\bigl(K_i \times M_i\bigr),$$
where for each $1 \leq i \leq t$, $K_i \cong \GL_{k_i}(q)$ is contained in the $\GL_j(q)$-factor of $\tw xM$, and 
$M_i \cong \GL_{m_i-k_i}(q)$ is contained in the $\GL_{n-j}(q)$-factor of $\tw xM$. Moreover, 
 $\prod^t_{i=1}K_i$ is a split Levi subgroup of $\GL_j(q)$, and $\prod^t_{i=1}M_i$ is a split Levi subgroup of $\GL_{n-j}(q)$. Now, applying 
 \cite[Lemma 2.7(i)]{GKNT} twice, we obtain
 $$\SR^{\tw x\cM}_{\cL \cap \tw x \cM}
    \bigl( (\alpha \boxtimes 1_{G_{n-j}})^x\bigr) = \SR^{\cG_j}_{K_1 \times \ldots \times K_t}(\alpha^x) \boxtimes 1_{M_1 \times \ldots \times M_t}.$$   
Recall that $\alpha^x$ is an irreducible character of $\GL_j(q)$. So, by \eqref{fn1}, the total sum of multiplicities of 
irreducible constituents $\beta$ in 
$\SR^{\tw x\cM}_{\cL \cap \tw x \cM}\bigl( (\alpha \boxtimes 1_{G_{n-j}})^x\bigr)$ is at most 
\begin{equation}\label{fn2}
  (j!)^2.
\end{equation}  
Consider any such irreducible constituent 
$$\beta = \alpha_1 \boxtimes \alpha_2 \boxtimes \ldots \boxtimes \alpha_t \boxtimes 1_{M_1 \times \ldots \times M_t}.$$
By \cite[Lemma 2.5(ii)]{GLT},
$$R^{\cL}_{\cL \cap \tw x \cM}(\beta) = \boxtimes^t_{i=1}R^{\cG_{m_i}}_{\cG_{k_i} \times \cG_{m_i-k_i}}\bigl(\alpha_i \boxtimes 1_{G_{m_i-k_i}}\bigr).$$
Let $\tau_{q,n}$ denote the permutation character of $G_n$ on $\F_q^n$, see \cite[(3.1)]{GLT}. Then the character
$$\gamma_i:=R^{\cG_{m_i}}_{\cG_{k_i} \times \cG_{m_i-k_i}}\bigl(\alpha_i \boxtimes 1_{G_{m_i-k_i}}\bigr)$$ 
is contained in $(\tau_{m_i,q})^{k_i}$ by
\cite[Proposition 3.2]{GLT}. 
If $k_i=0$, then the total number $N(\gamma_i)$ of multiplicities of irreducible constituents of $\gamma_i$ is 
$1$. If $1 \leq k_i \leq m_i/2$, then 
$$N(\gamma_i) \leq \langle \gamma_i,\gamma_i\rangle \leq \langle \tau_{m_i,q}^{k_i},\tau_{m_i,q}^{k_i}\rangle = 
   \langle \tau_{m_i,q}^{2k_i},1_{G_{m_i}}\rangle,$$
which is the number of $G_{m_i}$-orbits on ordered $2k_i$-tuples of vectors in $\F_q^{m_i}$, and hence is at most 
$8q^{k_i^2} \leq q^{4k_i^2}$ by \cite[Lemma 2.4]{GLT}. Suppose $m_i/2 < k_i \leq m_i$. Then $\gamma_i$ is a character of degree
at most $q^{m_ik_i} < q^{2k_i^2}$, and hence $N(\gamma_i) < q^{2k_i^2}$. Thus in all cases we have
$$N(\gamma_i) \leq q^{4k_i^2}.$$  
It follows that the total number $N(\beta)$ of multiplicities of irreducible constituents of
$$R^{\cL}_{\cL \cap \tw x \cM}(\beta) = \boxtimes^t_{i=1}\beta_i$$
is at most 
$$q^{4\sum^t_{i=1}k_i^2} \leq q^{4(\sum^t_{i=1}k_i)^2} = q^{4j^2}.$$
Combining this with \eqref{hc-n} and \eqref{fn2}, we have proved

\begin{prop}\label{hc-restr}
Let $G = \cG^F =\GL_n(q)$ and let $\chi$ be any irreducible character $G$ of true level $j \leq n/2$. If $L = \cL^F$ is a proper split Levi subgroup of 
$G$, then the total number $A$ of irreducible constituents (counting multiplicities) of the Harish-Chandra restriction 
$\SR^{\cG}_{\cL}(\chi)$ is at most $n^j(j!)^2q^{4j^2}$.
\end{prop}

\begin{cor}\label{hc-bound}
Let $G = \GL_n(q)$ and let $g \in G$ be any element
such that $\CB_G(g)$ is contained in a split Levi subgroup $\cL$ of $G$. Let $\chi \in \Irr(G)$ be of true level $j \leq n/2$, and let 
$D=\dim v^\cG$, with $v^\cG = \cO^*_\chi$ being the wave front set of $\chi$. Then 
$$|\chi(g)| \leq n^j(j!)^2q^{4j^2}\biggl( \frac{q+1}{q-1} \biggr)^{D/2}\chi(1)^{\alpha(L)}.$$
\end{cor}

\begin{proof}
%Note that the degree of any irreducible character of $\GL_n(q)$ is a monic polynomial in variable $q$ with integer coefficients. So in our case, 
As mentioned above, in our case the constant $B_1$ in \eqref{eq-d21} can be taken to be $1$. We now combine \eqref{eq-d20}, \eqref{eq-d21}, and \eqref{eq-d22} with Proposition \ref{hc-restr}.
\end{proof}

\subsection{The general case}
For semisimple elements whose centraliser is a non-split Levi subgroup, the bound in Corollary~\ref{hc-bound} can be very poor; 
for instance, it says nothing at all
about character values for elements in anisotropic tori.  
However, the following result is almost as good for all semisimple elements
as Corollary~\ref{hc-bound} is in the split case, and moreover works for both $\GL_n$ and $\GU_n$:

\begin{thm}\label{l-bound}
Let $G = \GL_n^{\epsilon}(q)$ and let $g \in G$ be any element 
such that $\CB_\cG(g)$ is contained in a proper $F$-stable Levi subgroup $\cL_1$. Define $L_1:=\cL_1^F$. Let $\chi \in \Irr(G)$ be of true level $j$, $0 \leq j \leq n$, and let 
$D=\dim v^\cG$, with $v^\cG = \cO^*_\chi$ being the wave front set of $\chi$. Then 
$$|\chi(g)| \leq n^{3j}\biggl( \frac{q+1}{q-1} \biggr)^{D/2}\chi(1)^{\alpha(L_1)}.$$
%where $L:=\CB_G(g)$.
\end{thm}

To prove this result we will use Deligne--Lusztig theory. However, before developing the necessary results about Deligne--Lusztig characters we recall a few facts about cosets. Assume $\mathfrak{G}$ is a group. The set of conjugacy classes of $\mathfrak{G}$ will be denoted by $\Cl(\mathfrak{G})$. A \emph{subcoset} of $\mathfrak{G}$ is a coset $Hw \subseteq \rN_{\mathfrak{G}}(H)$ of a subgroup $H \leqslant \mathfrak{G}$. Given any subsets $X,Y \subseteq \mathfrak{G}$ we define 
$$\rN_X(Y) := X \cap \rN_{\mathfrak{G}}(Y),~~\rC_X(Y) := X \cap \rC_{\mathfrak{G}}(Y),$$ 
where $\rN_{\mathfrak{G}}(Y)$ and $\rC_{\mathfrak{G}}(Y)$ are the usual normalizer and centralizer of $Y$. As usual 
$$XY := \{xy \mid x \in X \mbox{ and }y \in Y\}.$$

Now assume that $W\gamma \subseteq \mathfrak{G}$ is a finite subcoset. We denote by $\cf(W\gamma)$ the space of $W$-invariant functions $f : W\gamma \to \CC$, which we call class functions. This space has an inner product $\langle -,-\rangle$ and if $Hw \subseteq W\gamma$ is a subcoset then we have induction $\Ind_{Hw}^{W\gamma} : \cf(Hw) \to \cf(W\gamma)$ and restriction maps $\Res_{Hw}^{W\gamma} : \cf(W\gamma) \to \cf(Hw)$ which satisfy Frobenius reciprocity with respect to $\langle-,-\rangle$, see \cite[\S1.C]{Bon} or \cite[\S4]{TT}.

The function $\pi_w = \pi_w^{W\gamma}$ taking the value $\abs{\rC_W(w)}$ at any $W$-conjugate of $w \in W\gamma$ and the value $0$ otherwise is clearly contained in $\cf(W\gamma)$. We will need the following elementary calculation.

\begin{lem}\label{lem:res-coset}
For any subcoset $Hw \subseteq W\gamma$ and $x \in W\gamma$ we have
%%%%
\begin{equation*}
\Res_{Hw}^{W\gamma}(\pi_x^{W\gamma})
=
\sum_{\substack{z \in H\backslash W/\rC_W(x)\\ {}^zx \in Hw}} \frac{\abs{\rC_W(x)}}{\abs{\rC_H({}^zx)}}\pi_{{}^zx}^{Hw}
=
\sum_{\substack{z \in H\backslash W\\ {}^zx \in Hw}} \pi_{{}^zx}^{Hw}
\end{equation*}
%%%%
\end{lem}

\begin{proof}
The first equality is easy and the second follows because
%%%%
\begin{equation*}
Hz\rC_W(x) = \bigsqcup_{c \in \rC_H({}^zx)\backslash \rC_W({}^zx)} Hcz.\qedhere
\end{equation*}
%%%%
\end{proof}

We can also produce class functions in the following way. Consider the subgroup $W\langle \gamma\rangle \leqslant \rN_{\Gamma}(W)$ and let $\rho \in \Irr(W)$ be a $\gamma$-invariant irreducible character. The representation affording $\rho$ can be extended to a representation of $W\langle \gamma\rangle$ containing $\gamma^n$ in its kernel, for some $n > 0$. The trace function $\tilde\rho : W\langle \gamma\rangle \to \CC$ of such a representation is what we call an extension of $\rho$. Note the group $W\langle \gamma\rangle$ may be infinite but, by design, $\tilde\rho$ factors through a finite quotient.

The restriction $\Res_{W\gamma}^{W\langle \gamma\rangle}(\tilde\rho)$ of such an extension, which we usually again denote by $\tilde\rho$, is called an irreducible character of $W\gamma$. The set of irreducible characters is denoted by $\Irr(W\gamma)$. We say $\cB \subseteq \Irr(W\gamma)$ is a basis if it is a basis of $\cf(W\gamma)$. Every basis is orthonormal and is obtained by choosing for each $\gamma$-stable $\rho \in \Irr(W)$ exactly one extension to $W\langle \gamma\rangle$, see \cite[Proposition~11.6.3]{DM}. We need the following analogue of \cite[Corollary~4.11]{TT}

\begin{lem}\label{lem:res-bound}
Assume $Hw \subseteq W\gamma$ is a subcoset and $\rho_i \in \Irr(W)$, with $i \in \{1,2\}$, is $\gamma$-invariant. If $\tilde\rho_i$ is an extension of $\rho_i$ to $W\langle \gamma\rangle$ then
%%%%
\begin{equation*}
\abs{\langle \Res_{Hw}^{W\gamma}(\tilde\rho_1),\Res_{Hw}^{W\gamma}(\tilde\rho_2)\rangle}
\leqslant
\langle \Res_H^W(\rho_1),\Res_H^{W}(\rho_2)\rangle
\end{equation*}
%%%%
\end{lem}

\begin{proof}
Expanding out in a basis $\cB \subseteq \Irr(Hw)$ and using the triangle inequality
%%%%
\begin{align*}
\abs{\langle \Res_{Hw}^{W\gamma}(\tilde\rho_1),\Res_{Hw}^{W\gamma}(\tilde\rho_2)}
&\leqslant
\sum_{\tilde\eta \in \cB} \abs{\langle \tilde\eta, \Res_{Hw}^{W\gamma}(\tilde\rho_1)\rangle\langle \overline{\tilde\eta, \Res_{Hw}^{W\gamma}(\tilde\rho_2)\rangle}}\\
&=
\sum_{\tilde\eta \in \cB} \abs{\langle \Ind_{Hw}^{W\gamma}(\tilde\eta), \tilde\rho_1\rangle}\abs{\langle \Ind_{Hw}^{W\gamma}(\tilde\eta), \tilde\rho_2\rangle}.
\end{align*}
%%%%
Therefore, using \cite[Lem.~4.10]{TT} we obtain
%%%%
\begin{align*}
\abs{\langle \Res_{Hw}^{W\gamma}(\tilde\rho_1),\Res_{Hw}^{W\gamma}(\tilde\rho_2)}
&\leqslant
\sum_{\eta \in \Irr(H)} \langle \Ind_H^W(\eta), \rho_1\rangle \langle \Ind_H^W(\eta), \rho_2\rangle\\
&=
\sum_{\eta \in \Irr(H)} \langle \eta, \Res_H^W(\rho_1)\rangle \langle \eta, \Res_H^W(\rho_2)\rangle\\
&=
\langle \Res_H^W(\rho_1),\Res_H^W(\rho_2)\rangle.\qedhere
\end{align*}
\end{proof}

Recall that $\cG = \cG(\overline{\F_q}) = \cG^F$ is the group of $\overline{\F_q}$-points of a connected reductive $\FF_q$-group, with Frobenius $F$. We form the semidirect product $\cG\langle F\rangle$ with the infinite cyclic group generated by $F$, defined so that $FgF^{-1} = F(g)$ for all $g \in \cG$. If $\cH n \subseteq \cG F$ is a subcoset then the centralizer $\rC_{\cH}(n) \leqslant \rC_{\cG}(n)$ is a finite group. Moreover, if $\cH \leqslant \cG$ is closed and connected then, by the Lang--Steinberg Theorem, $\cH$ acts transitively by conjugation on $\cH n$. If $\cH$ is a Levi subgroup of $\cG$, resp., maximal torus of $\cG$, then we call $\cH n$ a \emph{Levi subcoset}, resp., a \emph{toral subcoset}.

We define
%%%%
\begin{equation*}
\scC(\cG F) := \{(g,n) \in \cG \times \cG F \mid gn = ng\}
\end{equation*}
%%%%
to be the set of commuting pairs. The group $\cG$ acts by simultaneous conjugation on $\scC(\cG F)$. We write $[g,n]$ for the orbit of $(g,n) \in \scC(\cG F)$ and $\scC(\cG F)/\cG$ for the set of orbits.

\begin{lem}\label{lem:GF-conj-classes}
The map $[g] \mapsto [g,F]$ is a well-defined bijection $\Cl(\rC_{\cG}(F)) \to \scC(\cG F)/\cG$.
\end{lem}

\begin{proof}
Clearly this is injective. If $(g,n) \in \scC(\cG F)$ then by the Lang--Steinberg Theorem $n = F^h$ for some $h \in \cG$ so $[g,n] = [{}^hg,F]$.
\end{proof}

Let $\cf(\scC(\cG F))$ be the set of $\cG$-invariant functions $f : \scC(\cG F) \to \mathbb{C}$. Via Lemma~\ref{lem:GF-conj-classes} we can identify $\cf(\scC(\cG F))$ with the space $\cf(\rC_{\cG}(F))$ of $\mathbb{C}$-valued class functions on the finite group $\rC_{\cG}(F)$. We define $\Irr(\scC(\cG F))$ to be those functions corresponding to $\Irr(\rC_{\cG}(F))$. The advantage of working with $\scC(\cG F)$ is that we can work with the different (inner) forms $\rC_{\cG}(gF)$ of $\cG$ simultaneously.

If $\cL w \subseteq \cG F$ is a Levi subcoset then we can define Deligne--Lusztig induction and restriction maps
%%%%
\begin{equation*}
R_{\cL w}^{\cG F} : \cf(\scC(\cL w)) \to \cf(\scC(\cG F))
\qquad\text{and}\qquad
\SR_{\cL w}^{\cG F} : \cf(\scC(\cG F)) \to \cf(\scC(\cL w)).
\end{equation*}
%%%%
For our purposes this can be done as follows. We start first with the case of a coset $\cL F$ where $\cL \leqslant \cG$ is an $F$-stable Levi subgroup. Making the identifications $\cf(\rC_{\cG}(F)) \to \cf(\scC(\cG F))$ and $\cf(\rC_{\cL}(F)) \to \cf(\scC(\cL F))$ we define 
$$R_{\cL F}^{\cG F} := R_{\cL}^{\cG},~~\SR_{\cL F}^{\cG F} := \SR_{\cL}^{\cG}.$$

Now consider a general Levi subcoset $\cL w \subseteq \cG F$. We pick an element $g \in \cG$ such that $F^g \in \cL w$, so that ${}^g(\cL w) = \cL_1F$ where $\cL_1 := {}^g\cL$ is an $F$-stable Levi subgroup of $\cG$. If $\iota_g : \cG\langle F\rangle \to \cG\langle F\rangle$ is the inner automorphism defined by $\iota_g(x) = {}^gx$ then $\iota_g(\cL w) = \cL_1F$ and we define
%%%%
\begin{equation*}
R_{\cL w}^{\cG F} := R_{\cL_1F}^{\cG F} \circ (\iota_g^{-1})^*
\qquad\text{and}\qquad
\SR_{\cL w}^{\cG F} := (\iota_g)^* \circ  \SR_{\cL_1F}^{\cG F},
\end{equation*}
%%%%
where $(\iota_g^{-1})^*$ is the map $f\mapsto f\circ\iota_g^{-1}$ and likewise for $(\iota_g)^*$.

We note that the maps $R_{\cL w}^{\cG F}$ and $\SR_{\cL w}^{\cG F}$ are defined only up to composition with $(\iota_n)^*$ for some $n \in \rN_{\cG}(\cL w) = \rN_{\rC_{\cG}(w)}(\cL)\cL$. We need the following interpretation of the Mackey formula.

\begin{lem}\label{lem:DL-at-ss}
If $\cL w \subseteq \cG F$ is a Levi subcoset and $\cT x \subseteq \cG F$ is a toral subcoset then
%%%%
\begin{equation*}
\SR_{\cL w}^{\cG F}\circ R_{\cT x}^{\cG F} = \sum_{\substack{z \in \cL \backslash \cG\\ {}^z(\cT x) \subseteq \cL w}} R_{{}^z(\cT x)}^{\cL w} \circ (\iota_z^{-1})^*
%\SR_{\cL w}^{\cG F}(R_{\cT x}^{\cG F}(\theta)) = \sum_{\substack{z \in W_{\cL}\backslash W\\ {}^zx \in W_{\cL}w}} R_{\cT {}^zx}^{\cL w}({}^z\theta)
\end{equation*}
%%%%
\end{lem}

\begin{proof}
Fix elements $h_w,h_x \in \cG$ such that $w = F^{h_w}$ and $x = F^{h_x}$ and let $\cL_1 := {}^{h_w}\cL$ and $\cT_1 := {}^{h_x}\cT$ be corresponding $F$-stable subgroups of $\cG$. According to the Mackey formula, see \cite[Thm~9.2.6]{DM}, we have
%%%%
\begin{equation*}
\SR_{\cL_1}^{\cG}\circ R_{\cT_1}^{\cG}
=
\sum_{\substack{u \in L_1\backslash G/T_1\\ {}^u\cT_1 \leqslant \cL_1}} T_{{}^u\cT_1}^{\cL_1} \circ (\iota_u^{-1})^*
=
\sum_{\substack{u \in L_1\backslash G\\ {}^u\cT_1 \leqslant \cL_1}} R_{{}^u\cT_1}^{\cL_1} \circ (\iota_u^{-1})^*,
\end{equation*}
%%%%
where $L_1 = \rC_{\cL_1}(F)$, $G = \rC_{\cG}(F)$, and $T_1 = \rC_{\cT_1}(F)$. The second equality follows because if ${}^u\cT_1 \leqslant \cL_1$ then ${}^uT_1 \leqslant L_1$ so
%%%%
\begin{equation*}
L_1uT_1 = L_1({}^uT_1)u = L_1u.
\end{equation*}

Consider the isomorphism of varieties $\psi : \cG \to \cG$ given by $\psi(v) = h_wvh_x^{-1}$. If $F' : \cG \to \cG$ is the morphism defined by $F'(v) = wvx^{-1}$ then we have $F\psi = \psi F'$. From this it follows that we have bijections
%%%%
\begin{equation*}
(\cL\backslash \cG)^{F'}
\overset{\psi}{\longrightarrow}
\rC_{\cL_1\backslash \cG}(F)
\longrightarrow
L_1 \backslash G,
\end{equation*}
%%%%
where $\rC_{\cL_1\backslash \cG}(F) = \{\cL_1 u \in \cL_1\backslash \cG \mid \cL_1 uF = \cL_1 Fu\}$, and $(\cL\backslash \cG)^{F'}$ denotes the cosets fixed by $F'$. The second bijection is a simple consequence of the Lang--Steinberg Theorem. If $z \in \cG$ then ${}^{\psi(z)}\cT_1 \leqslant \cL_1$ if and only if ${}^z\cT \leqslant \cL$ and $F'(\cL z) = \cL z$ if and only if ${}^zx \in \cL w$. It is clear that the combination of these two conditions is equivalent to the condition: ${}^z(\cT x) \subseteq \cL w$.

Finally, conjugating the expression above we get
%%%%
\begin{equation*}
\SR_{\cL w}^{\cG F}\circ R_{\cT x}^{\cG}
=
\sum_{\substack{z \in \cL\backslash \cG\\ {}^z(\cT x) \leqslant \cL w}} (\iota_{h_w})^* \circ R_{{}^{\psi(z)}\cT_1}^{\cL_1} \circ (\iota_{\psi(z)}^{-1})^* \circ (\iota_{h_x}^{-1})^*.
\end{equation*}
%%%%
It suffices to show that $(\iota_{h_w})^* \circ R_{{}^{\psi(z)}\cT_1}^{\cL_1} = R_{{}^z\cT}^{\cL} \circ (\iota_{h_w})^*$ when $F'(z) = z$, where $R_{{}^z\cT}^{\cL} = R_{({}^z\cT)w}^{\cL w}$ is defined with respect to the Frobenius $w$ on $\cL$. However, the arguments to prove this are identical to those used to prove \cite[Proposition~11.3.10]{DM}, see also the arguments by Bonnaf\'e in \cite{NTT}. We omit the details.
\end{proof}

From now on we assume the underlying group scheme is $\GL_n^{\epsilon}$ so that
$$\cG = \cG_n = \GL_n(\overline{\F_q}).$$
We will assume $\cT \leqslant \cG$ is the diagonal maximal torus and we denote by $W := \rN_{\cG}(\cT)/\cT \cong \SSS_n$ the corresponding Weyl group. The quotient $\rN_{\cG \langle F\rangle}(\cT)/\cT = W\langle F\rangle$ is isomorphic to the semidirect product $W\rtimes \langle F\rangle$, where we identify $F$ with its natural image $\cT F$. The coset $WF \subseteq W\langle F\rangle$ is, by definition, the set of toral subcosets $\cT n$ where $n \in \rN_{\cG F}(\cT)$.

We wish to reinterpret Lemma~\ref{lem:DL-at-ss} in the language of Lusztig's almost characters. Recall that $\GL_n^{\epsilon}$ is self-dual. If $(w,s) \in WF \times \cT$ is a pair such that $ws = sw$ then we set 
$$R_w^{\cG F}(s) := R_x^{\cG F}(\theta),$$ 
where $x \in WF$ and $\theta \in \Irr(\scC(w))$ correspond to $(w,s)$ under a bijection obtained as in \cite[Proposition~11.1.16]{DM} from duality.

If $s \in \cT$ and $\rC_{WF}(s) \neq \emptyset$ then, following Lusztig \cite[\S8.4]{Lus}, we define 
$$\cR_s^{\cG F} : \cf(\rC_{WF}(s)) \to \cf(\scC(\cG F)),~~
%by setting
%%%%
%\begin{equation*}
f \mapsto \frac{1}{\abs{\rC_W(s)}}\sum_{w \in \rC_{WF}(s)} f(w)R_w^{\cG F}(s) \in \cf(\scC(\cG F)).$$
%\end{equation*}
%%%%
where, as defined above, $\rC_X(s) = \{w \in X \mid ws=sw\}$ for any subset $X \subseteq W\langle F\rangle$.

\begin{cor}\label{cor:almost-at-ss}
Assume $\cT \leqslant \cL \leqslant \cG$ is a Levi subgroup and $w \in \rN_{WF}(\cL)$. Then for any $s \in \cT$ with $\rC_{WF}(s) \neq \emptyset$ we have
%%%%
\begin{equation*}
\SR_{\cL w}^{\cG F} \circ \cR_s^{\cG F}
=
\sum_{
	\substack{
		z \in H\backslash W/\rC_W(s)\\
		\rC_{Hw}({}^zs) \neq \emptyset}
	}
\cR_{{}^zs}^{\cL w}
\circ
\Res_{\rC_{Hw}({}^zs)}^{\rC_{WF}({}^zs)}
\circ
(\iota_z^{-1})^*
\end{equation*}
%%%%
where $H = \rN_{\cL}(\cT)/\cT$ is the Weyl group of $\cL$.
\end{cor}

\begin{proof}
By linearity it is enough to check both sides agree when evaluated at $\pi_x^{\rC_{WF}(s)}$ for some $x \in \rC_{WF}(s)$. But in that case $\cR_s^{\cG F}(\pi_x^{\rC_{WF}(s)}) = R_x^{\cG F}(s)$. Assume $(x,s)$ corresponds to $(y,\theta)$. Evaluating at $\theta$ we have by Lemma~\ref{lem:res-coset} that
%%%%
\begin{equation*}
\SR_{\cL w}^{\cG F}(R_y^{\cG F}(\theta))
=
\sum_{
	\substack{
		z \in \cL\backslash\cG\\
		{}^zy \subseteq \cL w}
	}
R_{{}^zy}^{\cL w}({}^z\theta).
\end{equation*}
%%%%

If ${}^zy \subseteq \cL w$ then ${}^z\cT \leqslant \cL$ so ${}^{lz}\cT = \cT$ for some $l \in \cL$. Therefore, we can take the sum over cosets $\rN_{\cL}(\cT)\backslash \rN_{\cG}(\cT)$ or similarly $H\backslash W$. As ${}^zy \subseteq \rN_{\cG F}(\cT)$ the condition ${}^zy \subseteq \cL w$ is equivalent to
%%%%
\begin{equation*}
{}^zy \subseteq \rN_{\cG F}(\cT) \cap \cL w = \rN_{\cL}(\cT)w
\end{equation*}
%%%%
which in turn is equivalent to ${}^zy \in Hw$. Breaking the sum in Lemma~\ref{lem:DL-at-ss} along double cosets, as in the proof of Lemma~\ref{lem:res-coset}, gives
%%%%
\begin{equation*}
\SR_{\cL w}^{\cG F}(R_y^{\cG F}(\theta))
=
\sum_{z \in H\backslash W/\rC_W(s)}
\sum_{
	\substack{
		c \in \rC_H({}^zs)/\rC_W({}^zs)\\
		{}^{cz}y \in Hw}
	}
R_{{}^{cz}y}^{\cL w}({}^{cz}\theta)
\end{equation*}
%%%%
Picking a different double coset representative we can assume that ${}^zy \in Hw$. We claim that $\rC_{Hw}({}^zs) = \rC_H({}^zs){}^zy$. Certainly ${}^zx \in \rC_{Hw}({}^zs)$ by assumption. Now if $d \in \rC_W({}^zs)$ then $d({}^zx) \in Hw$ if and only if $d \in \rC_H({}^zs) = H \cap \rC_W({}^zs)$. The statement now follows from Lemma~\ref{lem:res-coset}.
\end{proof}

We are now ready to prove Theorem~\ref{l-bound}.

\begin{proof}[Proof of Theorem~\ref{l-bound}]
By assumption, $\chi$ has true level $j$. Embed the maximal diagonal torus $\cT$ in the natural $F$-stable Levi subgroup
$\cM = \cM_1 \times \cM_2$, where $\cM_1\cong \cG_j$ and $\cM_2 \cong \cG_{n-j}$. Note that $F$ stabilises $\cM_1$ and $\cM_2$. We will identify $\cM\langle F\rangle$ with a subgroup of $\cM_1\langle F\rangle \times \cM_2\langle F\rangle$, where we again denote by $F$ its restriction to $\cM_i$.

By \cite[Theorem~3.9]{GLT} our assumption on the level of $\chi$ implies that $\chi$ is a constituent of
%%%%
\begin{equation*}
R_{\cM F}^{\cG F}(\alpha \boxtimes 1)
\end{equation*}
%%%%
for some $\alpha \in \Irr(\scC(\cM_1F))$. Note that
%%%%
\begin{equation*}
\rC_{\cM}(F) = \rC_{\cM_1}(F)\rC_{\cM_2}(F) \cong \GL_j^{\epsilon}(q) \times \GL_{n-j}^{\epsilon}(q)
\end{equation*}
%%%%
If $W_i \leqslant W$ is the subgroup $\rN_{\cM_i}(\cT)/\cT$ then the subgroup $W_1W_2 \leqslant W$ is a direct product with $W_1\cong \SSS_j$ and $W_2 \cong \SSS_{n-j}$.

By \cite[Theorem~11.7.3]{DM} we have $\alpha = \pm\cR_s^{\cM_1F}(\tilde\phi)$ for some $s \in \cM_1 \cap \cT$, with $\rC_{W_1F}(s) \neq \emptyset$, and some irreducible character $\tilde\phi \in \Irr(\rC_{W_1F}(s))$ afforded by a representation over $\QQ$, see \cite[Proposition~3.2]{Lus}. Note that $\rC_{W_1W_2}(s) = \rC_{W_1}(s)W_2$ is a reflection group and we have
%%%%
\begin{equation*}
\rC_{W_1W_2F}(s) = \rC_{W_1F}(s)W_2.
\end{equation*}
%%%%
It is known, see \cite[Proposition~11.6.6]{DM}, that
%%%%
\begin{equation*}
R_{\cM F}^{\cG F}(\alpha \boxtimes 1_{\cM_2F})
=
\pm\cR_s^{\cG F}(\Ind_{\rC_{W_1W_2F}(s)}^{\rC_{WF}(s)}(\tilde\phi\boxtimes 1))
\end{equation*}
%%%%
from which it follows that $\chi = \pm\cR_s^{\cG F}(\tilde\psi)$ for some irreducible constituent $\tilde\psi \in \Irr(\rC_{WF}(s))$ of the induced function $\Ind_{\rC_{W_1W_2F}(s)}^{\rC_{WF}(s)}(\tilde\phi\boxtimes 1)$.

We denote again by $\tilde\phi$ and $\tilde\psi$ irreducible characters of $\rC_{W_1}(s)\langle F\rangle$ and $\rC_W(s)\langle F\rangle$ respectively, yielding $\tilde\phi$ and $\tilde\psi$ upon restriction to the respective cosets. By \cite[Lemma~4.10]{TT}
%%%%
\begin{equation*}
\abs{\langle \Ind_{\rC_{W_1W_2F}(s)}^{\rC_{WF}(s)}(\tilde\phi\boxtimes1),\tilde\psi\rangle}
\leqslant
\langle \Ind_{\rC_{W_1W_2}(s)}^{\rC_W(s)}(\phi\boxtimes1), \psi\rangle.
\end{equation*}
%%%%
In particular, $\psi$ is a constituent of $\Ind_{\rC_{W_1}(s)W_2}^{\rC_W(s)}(\phi\boxtimes1)$.

Following the proof of Corollary~\ref{hc-bound} it suffices to show that
%%%%
\begin{equation}\label{eq-d222}
\sum_{\eta \in \Irr(L_1)} \abs{\langle \eta,\SR_{\cL_1 F}^{\cG F}(\chi)\rangle} \leqslant \langle \SR_{\cL_1 F}^{\cG F}(\chi),\SR_{\cL_1 F}^{\cG F}(\chi)\rangle \leqslant n^{3j}.
\end{equation}
%%%%
A straightforward argument shows that we may find a Levi subgroup $\cT \leqslant \cL \leqslant \cG$, an element $w \in \rN_{\cG F}(\cT)$, and an element $h \in \cG$ such that ${}^h(\cL w) = \cL_1F$, see \cite[Proposition~11.4.1]{DM}. With this we need only bound
%%%%
\begin{equation*}
\langle \SR_{\cL w}^{\cG F}(\chi),\SR_{\cL w}^{\cG F}(\chi) \rangle
=
\langle \SR_{\cL w}^{\cG F}(\cR_s^{\cG}(\tilde\psi)), \SR_{\cL w}^{\cG F}(\cR_s^{\cG}(\tilde\psi))\rangle.
\end{equation*}

By Corollary~\ref{cor:almost-at-ss}
%%%%
\begin{equation*}
\SR_{\cL w}^{\cG F}(\cR_s^{\cG}(\tilde\psi))
=
\sum_{
	\substack{
		z \in H\backslash W/\rC_W(s)\\
		\rC_{Hw}({}^zs) \neq \emptyset}
	}
	\cR_{{}^zs}^{\cL w}(
		\Res_{\rC_{Hw}({}^zs)}^{\rC_{WF}({}^zs)}({}^z\tilde\psi)
	),
\end{equation*}
%%%%
where $H = \rN_{\cL}(\cT)/\cT$ is the Weyl group of $\cL$. The number of double cosets appearing in this sum is bounded above by
%%%%
\begin{align*}
\abs{H\backslash W/\rC_W(s)} \leqslant \abs{W/\rC_W(s)} \leqslant \abs{W/W_2} = \abs{\SSS_n/\SSS_{n-j}} \leqslant n^j
\end{align*}
%%%%
because $W_2 \leqslant \rC_W(s)$.

Now, by the disjointness of Deligne--Lusztig characters, see \cite[Proposition~11.3.2]{DM}, the summands are pairwise orthogonal because each ${}^zs$ lies in a distinct $\cL$-conjugacy class. So, using Lemma~\ref{lem:res-bound} it suffices to bound
%%%%
\begin{equation*}
\abs{\langle \Res_{\rC_{Hw}({}^zs)}^{\rC_{WF}({}^zs)}(\tilde\psi),\Res_{\rC_{Hw}({}^zs)}^{\rC_{WF}({}^zs)}(\tilde\psi)\rangle}
\leqslant
\langle \Res_{\rC_H({}^zs)}^{\rC_W({}^zs)}(\psi),\Res_{\rC_H({}^zs)}^{\rC_W({}^zs)}(\psi)\rangle
\leqslant
\psi(1)^2
\end{equation*}
%%%%
But we know $\psi$ is a constituent of $\Ind_{\rC_{W_1}(s)W_2}^{\rC_W(s)}(\phi\boxtimes1)$ so
%%%%
\begin{align*}
\psi(1)^2 &\leqslant \Ind_{\rC_{W_1}(s)W_2}^{\rC_W(s)}(\phi\boxtimes1)(1)^2\\
&\leqslant \Ind_{\rC_{W_1}(s)W_2}^{\rC_W(s)}(\phi\boxtimes1)(1)^2\\
&\leqslant \abs{W/W_2}^2\\
& \leqslant n^{2j}.\qedhere
\end{align*}
\end{proof}

\subsection{An asymptotic version of \cite[Theorem 1.1]{BLST} and \cite[Theorem 1.9]{TT}} 
Now we can prove the main result of the section:

\begin{thm}\label{appr-bound}
For any $\epsilon > 0$, there are some explicit positive constants $N_0=N_0(\epsilon)$ and $q_0=q_0(\epsilon)$, such that the following 
statements holds for all integers $n \geq N_0$ and all prime powers $q \geq q_0$. 
Let $\cG = \GL_n(\overline{\F_q})$ and let $F:\cG\to\cG$ be a Frobenius 
endomorphism so that $\cG^F  \in \{\GL_n(q),\GU_n(q)\}$. Suppose we are in one of the following two cases.
\begin{enumerate}[\rm(i)] 
\item $G := \cG^F$ and $g \in G$ is any element such that $\CB_\cG(g)$ is contained in a proper $F$-stable Levi subgroup $\cL$ of $\cG$ and $L:= \cL^F$.
\item  $G:=[\cG,\cG]^F \in \{\SL_n(q),\SU_n(q)\}$, $g \in G$, and either 
\begin{enumerate}[\rm(a)]
\item $\CB_{[\cG,\cG]}(g)$ is contained in a proper split Levi subgroup $\cL$ of $[\cG,\cG]$ and $L:=\cL^F$, or 
\item $g$ is non-central semisimple with $L:=\CB_G(g)$. 
\end{enumerate}
\end{enumerate}
Then  
\begin{equation}\label{eq-bound1}
  |\chi(g)| \leq \chi(1)^{\alpha(L)+\epsilon}
\end{equation}  
for all $\chi \in \Irr(G)$.
\end{thm}

\begin{proof}
Note that \eqref{eq-bound1} is obvious if $\alpha:=\alpha(L)\geq 1-\epsilon$. So in what follows we will assume
$$\alpha+\epsilon < 1,$$
in particular $0 < \epsilon < 1$. 

\smallskip
(A) First we prove \eqref{eq-bound1} in the cases of (i) and (ii)(a). Note that the upper bound on $|\chi(g)|$ in Theorem \ref{l-bound} is obtained by combining 
\eqref{eq-d20}, \eqref{eq-d21}, and \eqref{eq-d222}. If $\cL$ is split, $\varphi:=\SR_{\cL_1 F}^{\cG F}(\chi)$ is a true character of $\cL_1^F$
(in the notation of the proof of Theorem \ref{l-bound}), hence this upper bound is actually an upper bound on the degree of $\varphi$. Arguing as in part (ii) of the proof of \cite[Theorem 1.1]{BLST}, it therefore suffices to prove \eqref{eq-bound1} for
$$G \in \{\GL_n(q),\GU_n(q)\}.$$
%In the case $\cL$ is split, we may assume that $L$ is as given in \eqref{levi-l}. 

\smallskip
Let $j=\cl(\chi)$ be the \emph{level} of $\chi$ in the sense of \cite[Definition 1(ii)]{GLT}. This means that multiplying $\chi$ by a suitable linear character of $G$, we may assume that $\tcl(\chi)=j$. 
Applying Theorem \ref{l-bound}, it suffices to prove 
\begin{equation}\label{eq-bound3}
  n^{3j}\biggl( \frac{q+1}{q-1} \biggr)^{D/2} \leq \chi(1)^{\epsilon}.
\end{equation}

Note that the degree of any irreducible character of $G$ is a monic polynomial in variable $q$ with integer coefficients,
in fact a product of a power of $q$ and cyclotomic polynomials in $q$. 
Writing $D=\dim v^\cG$, with $v^\cG = \cO^*_\chi$ being the wave front set of $\chi$, we therefore have
%\begin{equation}\label{eq-d30}
$$\chi(1) \geq (q-1)^{D/2}.$$
%\end{equation}  
%As mentioned above, in our case the constant $B_1$ in \cite[Lemma 2.4]{BLST} can be taken to be $1$. Hence, the arguments in part (i) of the proof of \cite[Theorem 1.1]{BLST}
%show that 
%where $A$ is the total number of irreducible constituents of $\rho:=\SR^G_L(\chi)$.
%% and $\alpha:=\alpha(L)$. 
%In the case $\cL$ is non-split, by Theorem \ref{l-bound} we can take $A=n^{3j}$ instead.
Choosing $q_0=q_0(\epsilon) \geq 3$ such that 
$$\frac{q_0+1}{q_0-1}  \leq (q_0-1)^{\epsilon/3},$$
it remains to prove  
\begin{equation}\label{eq-bound4}
  n^{3j} \leq \chi(1)^{2\epsilon/3}
\end{equation}
for $q \geq q_0$ and $n \geq N_0$.

Choosing $N_0 \geq 4$ we have $n^2/4-2 > n^2/16$ for $n \geq N_0$, and so, when $j > n/2$ we have
$$\chi(1) > q^{n^2/16} > q^{n^2\epsilon/16}$$
by \cite[Theorem 1.2(ii)]{GLT}. Next, if $n/4 < j \leq n/2$, then $j(n-j) > n^2/8$, and so
$$\chi(1) > q^{n^2/8}$$
by  \cite[Theorem 1.2(i)]{GLT}. If $0 \leq j \leq n/4$, then $j(n-j) \geq 3nj/4$, and so 
\begin{equation}\label{for-c1}
  \chi(1) \geq q^{3nj/4}
\end{equation}
again by \cite[Theorem 1.2(i)]{GLT}.  

\smallskip
First we work in the regime 
$$n\epsilon/12 \leq j \leq n.$$
Then \eqref{for-c1} and the above arguments show that 
$$\chi(1) \geq q^{n^2\epsilon/16}.$$
Choose $N_0 \geq 4$ such that 
\begin{equation}\label{for-n1}
 n \leq q_0^{n\epsilon^2/72}
\end{equation} 
for all $n \geq N_0$. Then for $q \geq q_0$ 
we now have 
%by \eqref{fn1} 
%in the case $L$ is split, and by \cite[Theorem 1.9]{TT} in the non-split case, 
$$n^{3j} \leq n^{3n} \leq q_0^{n^2\epsilon^2/24} \leq \chi(1)^{2\epsilon/3},$$
yielding \eqref{eq-bound4} in this case.

Assume now that $j \leq n\epsilon/12 \leq n/12$. Then 
%we choose $N_0 \geq 4$ such that 
%$$\frac{N_0^3\epsilon^2}{144} \leq q_0^{N_0\epsilon/6}.$$
%Then 
for $n \geq N_0$ and $q \geq q_0$ we now have by \eqref{for-c1} and \eqref{for-n1}
% Proposition \ref{hc-restr}, and Theorem \ref{l-bound} 
that
$$n^{3j} \leq q_0^{n\epsilon^2 j/24}  <q^{n\epsilon j/2} \leq \chi(1)^{2\epsilon/3},$$
proving \eqref{eq-bound4} in this case as well.

\smallskip
(B) Now we handle the case (ii)(b), embedding $G$ in $\tilde G:=\cG^F$. Letting $\tilde L := \CB_G(g)$, note that $\tilde G = G\tilde L$ and $g \in \tilde L$. Letting $\tilde \chi \in \Irr(\tilde G)$ lie above $\chi$, by Clifford's theorem we have
$$\tilde \chi|_G = \sum^t_{i=1}\chi^{x_i},$$
where $x_1, \ldots,x_t$ can be chosen from $\tilde L$. Since every $x_i$ centralizes $g$, we have
$$\tilde \chi(1) = t\chi(1),~~\tilde\chi(g) = t\chi(g).$$
By the case (i) proved in (A), $|\tilde\chi(g)| \leq \tilde\chi(1)^{\alpha+\eps}$. As $\alpha+\eps <1$, it follows that $|\chi(g)| \leq \chi(1)^{\alpha+\eps}$.
\end{proof}

\section{Some numerical estimates}

This section is devoted to numerical estimates which allow us to determine a finite list of possible exceptions to Conjecture \ref{Form}.

Let 
$$f_{a,x}(\delta):= \min\Bigl(\bigl(\frac 1a+\delta\bigr)x,\frac 12\bigl(\frac 1a+ \frac a{a-1}\delta^2\bigr)\Bigr).$$
%
%In our application,  $x$ will be $\log_{q^{n^2}}\!\chi(1)$, where $\chi$ is a non-linear irreducible character,
%and $f_{a,x}(\delta)$ is an upper bound on $\log_{q^{n^2}}\!|\chi(t)|$, where $t$ is a semisimple element of order $a$ (and where $\delta$ quantifies how far eigenvalue multiplies are from being equal).  
%We express $f_{a,x}(\delta)$ as a minimum because it incorporates
%two competing upper bounds: the centralizer bound, which holds for all elements of all finite groups, 
%and the bound from Proposition~\ref{l-bound}, which is specific to semisimple elements in groups of Lie type.
The main goal of this section is to 
give an explicit finite list of possible exceptions to the rule that 
if $2\le a_1 \le a_2\le \cdots \le a_r$ are integers such that $\sum_{i=1}^r \frac 1{a_i} < r-2$, then
$$\sum_{i=1}^r f_{a_i,x}(\delta_i) < x + \sum_{i=1}^r \frac{a_i\delta_i^2}{a_i-1}$$
when $x\in [0,1/2]$ and
$\delta_i\in [0,\frac{a_i-1}{a_i}]$ for all $i$.
We will see that when the rule holds, Conjecture~\ref{Form} holds for the corresponding genus $0$ Fuchsian group.  There is only one exception with $r\ge 4$ and thirty-one with $r=3$.

\begin{prop}
For all $a\ge 2$, $x\in (0,1/2]$, and $\delta\in [0,1-1/a]$, the function
$$\frac{f_{a,x}(\delta) - \frac {a\delta^2}{a-1}}x$$
is bounded above by $\max\bigl(\frac xa, G_{a,x},H_{a,x}\bigr)$, where
%
%$$\frac 1a+\frac{\sqrt{3a+1}-2}{6a}$$
$$G_{a,x}
= \begin{cases}
\frac{(a-1)x+4}{4a},&\text{if } x\le\frac{2\sqrt{3a+1}-4}{3(a-1)}\\
-\infty,&\hbox{if } x>\frac{2\sqrt{3a+1}-4}{3(a-1)}
\end{cases}$$
and
$$H_{a,x} 
= \begin{cases}
\frac{\sqrt{(a-1)^2 x^2 + 2(a-1)x - (a-1)}}a  - \frac{(a-1)x}a + \frac {1-x}{ax},&\text{if }(a-1)x^2+ax\ge 1\\
-\infty,&\text{if }(a-1)x^2+ax < 1.
\end{cases}$$
In particular,
\begin{equation}
\label{dumb}
f_{a,x}(\delta) < \frac {2x}{\sqrt a}+\frac {a\delta^2}{a-1}.
\end{equation}
\end{prop}

\begin{proof}
Let
\begin{equation*}
\begin{split}
g_{a,x}(\delta) &:= \bigl(\frac 1a+\delta\Bigr) x - \frac {a\delta^2}{a-1},\\
h_a(\delta) &:=  \frac 1{2a}- \frac{a\delta^2}{2(a-1)},\\
\end{split}
\end{equation*}
so
$$f_{a,x}(\delta) - \frac {a\delta^2}{a-1} = \min(g_{a,x}(\delta),h_a(\delta)).$$
For each fixed integer $a\ge 2$, we wish to determine as a function of $x\in (0,1/2]$, the (unique) element $\delta_0(x)\in \bigl[0,\frac{a-1}a\bigr]$ for which $\min(g_{a,x}(\delta),h_a(\delta))$ achieves its maximum as a function of $\delta$.  

We note first that as functions of $\delta$, $g_{a,x}(\delta)$ and  $h_a(\delta)$ are strictly concave, and
$h_a(\delta)$ is decreasing on $[0,\infty)$.
Therefore, $\delta_0(x)$ must either be the unique critical point $\frac{(a-1)x}{2a}$ of $g_{a,x}(\delta)$,
the minimum solution of
$g_{a,x}(\delta) = h_a(\delta)$
in the interval $\bigl[0,\frac{a-1}a\bigr]$, or one of the endpoints $0$ and $\frac{a-1}a$ of the interval.  For the endpoints, we have
$$\min(g_{a,x}(0),h_a(0)) \le g_{a,x}(0) = \frac xa,$$
and
$$\min\Bigl(g_{a,x}\bigl(\frac{a-1}a\bigr),h_a\bigl(\frac{a-1}a\bigr)\Bigr) \le h_a\bigl(\frac{a-1}a\bigr) = \frac 1a-\frac 12\le 0.$$

If the maximum occurs at $\frac{(a-1)x}{2a}$, it must be 
$$g_{a,x}\Bigl(\frac{(a-1)x}{2a}\Bigr) = \frac{(a-1)x^2 + 4x}{4a},$$
and this quantity must be less than or equal to 
$$h_a\Bigl(\frac{(a-1)x}{2a}\Bigr) = \frac{4-(a-1)x^2}{8a},$$
so
$$x \le \frac{2\sqrt{3a+1}-4}{3(a-1)}.$$
Thus, for all $x$,
\begin{equation}
\label{first ineq}
G_{a,x} \le \frac 1a+\frac{\sqrt{3a+1}-2}{6a}\le \frac 1a + \frac1{\sqrt{2a}}<\frac 2{\sqrt a}.
\end{equation}

The graphs of $g_{a,x}(\delta)$ and $h_a(\delta)$ intersect 
only if 
\begin{equation}
\label{real roots}
(a-1)x^2+2x\ge 1,
\end{equation}
in which case the
smaller $\delta$-value satisfying  $g_{a,x}(\delta) = h_a(\delta)$ is
$$\delta = \frac{(a-1)x- \sqrt{(a-1)^2 x^2 + 2(a-1)x - (a-1)}}a.$$
If this is $\delta_0(x)$, we have
\begin{equation}
\label{gax}
H_{a,x} = \frac{\sqrt{(a-1)^2 x^2 + 2(a-1)x - (a-1)}x}a  - \frac{(a-1)x^2}a + \frac {1-x}a\le \frac {1-x}a\le \frac 1a.
\end{equation}
By \eqref{real roots}, $x>\frac 1{2\sqrt a}$, so \eqref{gax} implies
$$\frac{g_{a,x}(\delta_0(x))}x < \frac 2{\sqrt a}.$$
Together with \eqref{first ineq}, this implies the proposition.
\end{proof}

\begin{prop}
\label{Genus 0}
Let $r\ge 3$ be an integer, $a_1\le a_2\le \cdots \le a_r$ be integers $\ge 2$, and $\delta_i$ be non-negative numbers with $\frac 1{a_i}+\delta_i\le 1$ for all $i$.   We assume that the product $a_1\cdots a_r$ is not on the following list:
\begin{enumerate}[\rm(i)]
\item $23c$, $7\le c\le 24$
\item $24c$, $5\le c\le 9$
\item $25c$, $5\le c\le 7$
\item $266$,
\item $33c$, $4\le c\le 6$
\item $344$,
\item $2223$.
\end{enumerate}
Then for all $x\in (0,1/2]$, 
$$\sum_{i=1}^r f_{a_i,x}(\delta_i) < -\epsilon x+(r-2)x+\sum_{i=1}^r \frac{a_i \delta_i^2}{a_i-1},$$
where $\epsilon$ is a positive constant which does not depend on $x$, $r$, the $a_i$, or the $\delta_i$.
\end{prop}

\begin{proof}
By \eqref{dumb} for $a\ge 100$ and machine computation for $2\le a< 100$,
\begin{equation}
\label{f bound}
\frac{f_{a,x}(\delta) - \frac {a\delta^2}{a-1}}x\le
\begin{cases}
.555&\text{if }a=2,\\
.399&\text{if }a=3,\\
.318&\text{if }4\le a < 100,\\
.2&\text{if }100\le a < 10000,\\
.02&\text{if }10000\le a.
\end{cases}
\end{equation}
Therefore, if $r\ge 5$,
$$r-2-\sum_{i=1}^r \frac{f_{a_i,x}(\delta_i) - \frac {a_i\delta_i^2}{a_i-1}}x > r-2-.56\,r \ge .2.$$
If $r= 4$ and $a_4\ge 4$, then
$$r-2-\sum_{i=1}^r \frac{f_{a_i,x}(\delta_i) - \frac {a_i\delta_i^2}{a_i-1}}x > 2 - .56\cdot 3-.318 = .002,$$
while if $r=4$ and $a_3 \ge 3$, then
$$r-2-\sum_{i=1}^r \frac{f_{a_i,x}(\delta_i) - \frac {a_i\delta_i^2}{a_i-1}}x > 2 - .56\cdot 2-.4\cdot 2 = .08.$$
The only remaining possibility for $r=4$ is $2223$.

For $r=3$, we may assume $a_2\ge 3$, so if $a_3 \ge 10000$,
$$r-2-\sum_{i=1}^r \frac{f_{a_i,x}(\delta_i) - \frac {a_i\delta_i^2}{a_i-1}}x > 1 - .56-.4-.02 = .02.$$
The triples with $a_3<10000$ can be handled exhaustively by machine, by partitioning the $x$-interval $[0,1/2]$ into
subintervals on which $G_{x,a}$ and $H_{x,a}$ are bounded above.
\end{proof}

We remark that it is feasible to give a proof of a somewhat weaker result which does not require machine computation at all.
An earlier version of this paper gave an explicit argument which holds for all triangle groups of type $abc$ except the following:

\begin{enumerate}[\rm(a)]
\item $a=2$, $b=3$, $7\le c\le 295$
\item $a=2$, $b=4$, $5\le c\le 26$
\item $a=2$, $b=5$, $5\le c\le 15$
\item $a=2$, $6=6$, $6\le c\le 11$
\item $(a,b,c) \in \{(2,7,7),(2,7,8),(2,7,9), (2,7,10), (2,8,8)\}$
\item $a=b=3$, $4\le c\le 11$
\item $(a,b,c)\in \{(3,4,4),(3,4,5),(3,4,6),(3,4,7),(3,5,5),(4,4,4)\}$
\end{enumerate}

It holds, moreover, for $r=4$ except in the case of groups of type $2223$ and $2224$ and for all groups with $r\ge 5$.

\begin{lem}
\label{Positive Genus}
Let $g$ be a positive integer and $r$ a non-negative integer such that if $g=1$, then $r>0$.  Let
$a_1\le a_2\le\cdots \le a_r$ be a (possibly empty) sequence of integers $\ge 2$.  Then for all $x\in \bigl[0,\frac 12\bigr]$, 
$$\sum_{i=1}^r f_{a_i,x}(\delta_i) < -.44 x +(2g+r-2)x+\sum_{i=1}^r \frac{a_i \delta_i^2}{a_i-1}.$$
\end{lem}

\begin{proof}
If $g=1$ and $r\ge 1$,
\eqref{f bound} implies
$$\sum_{i=1}^r \frac{f_{a_i,x}(\delta_i) - \frac {a_i\delta_i^2}{a_i-1}}x <.56 r\le -.44+r = -.44+(2g+r-2).$$
If $g\ge 2$ and $r\ge 0$,
\begin{equation*}
\sum_{i=1}^r \frac{f_{a_i,x}(\delta_i) - \frac {a_i\delta_i^2}{a_i-1}}x \le.56 r. < -.44 + (2+r) \le -.44 + (2g+r-2).\qedhere
\end{equation*}
\end{proof}

\section{Asymptotics of $j_{q,n}(a)$}

Let $t$ be an element of $G_n = \GL_n(q)$ of order $a$.  We assume $q$ is prime to $a$,  Let $\zeta=\zeta_a$
be a primitive $a$th root of unity in $\bar\F_q$,
so $z^a-1 = 0$ has distinct roots $\zeta,\zeta^2,\ldots,\zeta^a = 1$, in $\bar\F_q$.
Let $m_i$ denote the multiplicity of $\zeta^i$ as an eigenvalue of $t$.  We write $i\sim j$ if
$\zeta^i$ and $\zeta^j$ have the same Frobenius orbit.  The vector $(m_1,\ldots,m_a)$ satisfies
the following conditions:

\begin{enumerate}
\item $m_i\in\Z$ for all $i$,
\item $m_1+\cdots+m_a = n$,
\item $m_i = m_j$ whenever $i\sim j$,
\item $m_i\ge 0$ for all $i$.
\end{enumerate}
The element $t$ is determined up to conjugacy in $\GL_n(q)$ by the vector $(m_1,\ldots,m_a)$.
For given $n$, the vector is determined by $m_1,\ldots,m_{a-1}$, so the number of possibilities is $O(n^{a-1})$.

Let $S$ denote the subset of  $\{1,\ldots,a\}$ consisting of the smallest element in each Frobenius orbit, and let $l_s$ be the size of
the orbit of $s$.  
The centralizer of $t$ in $\GL_n(q)$ can be written $\prod_{s\in S}\GL_{m_s}(q^{l_s})$, so the conjugacy class $C = t^{G_n}$
satisfies
\begin{equation}
\label{Class size}
|C| = \frac{q^{n^2}\prod_{j=1}^n (1-q^{-j})}{\prod_{s\in S} \Bigl(q^{l_s m_s^2}\prod_{j=1}^{m_s}(1-q^{-l_s j})\Bigr)}
= \frac{q^{n^2 - \sum_{i=1}^a m_i^2}q^{(1-a)/24}\eta(\frac{i\log q}{2\pi})}{\prod_{s\in S}\eta(\frac{i l_s \log q}{2\pi})}
\frac{\prod_{s\in S}\prod_{j=m_s+1}^\infty (1-q^{-l_s j})}{\prod_{j=n+1}^\infty (1-q^{-j})},
\end{equation}
where $\eta(z)$ is the Dedekind eta-function.
The second multiplicand on the right hand side can be bounded above in terms of $a$, and it approaches $1$ as $\inf_i m_i$ goes to $\infty$.

Writing
\begin{equation}
\label{uneven}
n^2 - \sum_{i=1}^a m_i^2 = n^2\bigl(1-\frac 1a\bigr)-\sum_{i=1}^a \bigl(\frac na-m_i\bigr)^2,
\end{equation}
we see that if $m_i \le n/2a$, then 
$$|C| = O(q^{(1-1/a-1/4a^2)n^2}),$$
so the sum of $|C|$ over all conjugacy classes with $\inf_i m_i \le n/2a$ is $o(q^{(1-1/a)n^2})$.  
%As
%%
%$$\sum_{j=n+1}^\infty \log \infty (1-x^j)| < 4 x^j$$
%%
%for $|x| \le 1/2$, defining
%%
%$$f(z) = \frac{\eta(z)
%%
%$$\log |C| = (n^2-\sum_{i=1}^a m_i^2)\log q + \log f(

We define $j_{q,n,k}(a)$ to be the number of elements $t\in \GL_n(q)$ with $t^a=1$ and $\det(a) = \zeta^k$.
Consider the subset of $\Z^a$ satisfying conditions (1)--(3) and the congruence condition
\begin{equation}
\label{fixed det}
\sum_{i=1}^r i m_i \equiv k\pmod a,
\end{equation}
which is equivalent to the condition $\det(t) = \zeta^k$.  This is a coset $\lambda_n+\Lambda$, where $\Lambda$ 
is a subgroup of $\Z^a$ which does not depend on $n$ and $\lambda_n\in \Z^a$ has coordinate sum $n$.
Moreover, adding $2$ to each $m_i$ and $2a$ to $n$ preserves the sets satisfying
conditions (1)--(3) and \eqref{fixed det}, so
$$\lambda_{n+2a} = \lambda_n + (2,2,\ldots,2).$$
Thus,  
$$\lambda'_n := \lambda_n - \bigl(\frac na,\frac na,\ldots,\frac na\bigr),$$
is periodic in $n$ with period $2A$ and has coordinate sum $0$.

By \eqref{Class size} and \eqref{uneven},
\begin{align*}
j_{q,n,k}(a) &= \sum_{(m_1,\ldots,m_a)\in (\lambda_n+\Lambda)\cap \N^a}\frac{q^{n^2}\prod_{j=1}^n (1-q^{-j})}{\prod_{s\in S} \Bigl(q^{l_s m_s^2}\prod_{j=1}^{m_s}(1-q^{-l_s j})\Bigr)}\\
&= \sum_{(m_1,\ldots,m_a)\in (\lambda_n+\Lambda)\cap \N^a}
\frac{q^{n^2 - \sum_{i=1}^a m_i^2}q^{(1-a)/24}\eta(\frac{i\log q}{2\pi})}{\prod_{s\in S}\eta(\frac{i l_s \log q}{2\pi})}
+ o(q^{(1-1/a)n^2})\\
&= \frac{q^{(1-a)/24}\eta(\frac{i\log q}{2\pi})}{\prod_{s\in S}\eta(\frac{i l_s \log q}{2\pi})}
\sum_{(m_1,\ldots,m_a)\in \lambda_n+\Lambda}
q^{n^2 - \sum_{i=1}^a m_i^2}
+ o(q^{(1-1/a)n^2})\\
&= \frac{q^{(1-a)/24}\eta(\frac{i\log q}{2\pi})}{\prod_{s\in S}\eta(\frac{i l_s \log q}{2\pi})} q^{(1-1/a)n^2}
\sum_{\lambda'\in \lambda'_n+\Lambda} q^{-\lambda'\cdot \lambda'} + o(q^{(1-1/a)n^2}),
\end{align*}
where the implicit constant on the right hand side does not depend on $q$.
Defining 
$$\theta_v(z) := \sum_{\lambda\in v+\Lambda} e^{2\pi i(\lambda\cdot\lambda) z}$$
and 
$$f_n(z) := \frac{\eta(z)\theta_{\lambda'_n}(z)}{\prod_{s\in S} \eta(l_s z)},$$
we have proved the following proposition:

\begin{prop}
\label{J nqk}
The periodic sequence $f_1,f_2,f_3,\ldots$ 
of half-integral weight modular forms with integral $q$-expansions satisfies
$$j_{q,n,k}(a) = \bigl(f_n\bigl(\frac{i\log q}{2\pi}\bigr)+o(1)\bigr)q^{(1-a)/24}q^{(1-1/a)n^2},$$
where the $o(1)$ term does not depend on $q$.
\end{prop}

From this, it is easy to deduce:
\begin{cor}
\label{J of ai}
Let $a_1,\ldots,a_r$ denote positive integers with least common multiple $A$.  Then there exists a $2A$-periodic
sequence of meromorphic modular forms $f_1,f_2,f_3,\ldots$ with integral Fourier coefficients, holomorphic except possibly at $\infty$, such that
$$J_{q,n}(a_1,\ldots,a_r) 
= \bigl(f_n\bigl(\frac{i\log q}{2\pi}\bigr)+o(1)\bigr) q^{(r-a_1-\cdots-a_r)/24} q^{(r-1/a_1-\cdots-1/a_r)n^2}.$$
\end{cor}

\begin{proof}
Let 
$$\Sigma(a_1,\ldots,a_r) := \{(k_1,\ldots,k_r)\in (\Z/a_1\Z)\times \cdots\times (\Z/a_r\Z)\mid \prod \zeta_{a_i}^{k_i}=1\}.$$
If $t_i^{a_i}=1$, then $\det(t_i) = \zeta_{a_i}^{k_i}$ for a well-defined $k_i\in \Z/a_i\Z$.
Every element $(t_1,\ldots,t_r)$ of \eqref{J set} determines $(k_1,\ldots,k_r)\in\Sigma(a_1,\ldots,a_r)$
such that $\det(t_i) = \zeta_{a_i}^{k_1}$.  Therefore,
$$
J_{q,n}(a_1,\ldots,a_r) = \sum_{(k_1,\ldots,k_r)\in \Sigma(a_1,\ldots,a_r)} \prod_{i=1}^r j_{q,n,k_i}(a_i),
$$
and the corollary follows immediately from Proposition~\ref{J nqk}.
\end{proof}

\section{Counting Fuchsian group representations}

If $t^a=1$, and $m_1,\ldots,m_a$ are the eigenvalue multiplicities of $t$, define
$\delta := -1/a + \sup_i m_i/n$.  Thus, $\delta\ge 0$.  Let $j$ be chosen so $m_j = n/a+\delta n$.

\begin{prop}
\label{Combined bound}
Let $q$ be any sufficiently large prime power.
For all $\epsilon > 0$, there exists $N$ such that if $n>N$, if $x\in (0,1/2]$ and $\chi$ is an irreducible character of $G_n$ of degree $q^{xn^2}>1$, then
%
%$$\frac{\log |\chi(t)|}{n^2\log q} \le \min\Bigl(\frac 1{2a}+\frac{a^2\delta}{2(a-1)},\bigl(\frac 1a+\delta+\epsilon\bigr)x\Bigr).$$
$$\frac{\log |\chi(t)|}{n^2\log q} \le f_{a,x}(\delta) + o(x).$$
\end{prop}

\begin{proof}
As 
$$\sum_{i\neq j} \bigl(\frac na-m_i\bigr) \ge \delta n,$$
by the Cauchy-Schwartz inequality.
$$\sum_{i\neq j} \bigl(\frac na-m_i\bigr)^2 \ge \frac{\delta^2 n^2}{a-1}.$$
By \eqref{uneven}, the dimension of the centralizer of $t$ in the algebraic group $\GL_n$ is
\begin{equation}
\label{centralizer dim}
\frac{n^2}a + \sum_i (\frac na-m_i)^2 \ge \frac{n^2}a + \delta^2 n^2 + (a-1)\frac{\delta^2 n^2}{(a-1)^2}
= \frac{n^2}a + \frac{a\delta^2 n^2}{a-1}.
\end{equation}
The centralizer of $t$ in $G_n$ is less than or equal to $q$ to the power of the centralizer dimension, so
the centralizer bound implies that for every irreducible character $\chi\in \Irr(G_n)$,
\begin{equation}
\label{centralizer}
|\chi(t)| \le q^{n^2(\frac 1{2a}+\frac{a\delta^2}{2(a-1)})}.
\end{equation}

On the other hand, by \cite[Theorem~1.10]{BLST}, if $L$ denotes the centralizer of $t$,
$$\alpha(L) \le \frac{\sup_i m_i}n = \frac 1a + \delta.$$
The character bound Theorem~\ref{appr-bound} therefore implies
$$|\chi(t)| \le \chi(1)^{\frac 1a+\delta + o(1)} = q^{n^2x \bigl(\frac 1a+\delta + o(1)\bigr)}.$$
Combining this with \eqref{centralizer}, we obtain the proposition.
\end{proof}

\begin{prop}
\label{non-triv char}
Let $q$ be any sufficiently large prime power.
There exists $\epsilon > 0$  such that for all $r\ge 3$ and all $2\le a_1\le a_2\le \cdots\le a_r$  not excluded in Proposition~\ref{Genus 0} above, there exists $N$ such that 
if $q$ is prime to all $a_i$,
$n>N$, the elements $t_i\in G_n$ satisfy $t_i^{a_i}=1$, and $\chi$ is a non-linear  irreducible character of $G_n$, we have
$$\frac{\prod_i \bigm|\!t_i^{G_n}\!\!\bigm|}{|G_n|}\frac{\prod_i |\chi(t_i)|}{\chi(1)^{r-2}}
\le q^{n^2\bigl(-1+\sum_i (1-\frac 1{a_i})\bigr)}\chi(1)^{-\epsilon}.$$
\end{prop}

\begin{proof}
Let $x := \log_{q^{n^2}}\chi(1)$, which by \cite{LS} is at least $1/2n$.
Let $\delta_i:=\mu_i-1/a_i$, where $\mu_i n$ is the highest multiplicity of any eigenvalue of $t_i$,
$Z_i$ be the centralizer of $t_i$ in $G_n$, and $d_i n^2$ be the dimension of 
the centralizer of $t_i$ in the algebraic group $\GL_n$.  By \eqref{centralizer dim}, for all $i$,
$$d_i \ge \frac 1{a_i}+\frac{a_i\delta_i^2}{a_i-1},$$
and
$$|Z_i| \ge  q^{d_in^2}\prod_{i=1}^\infty (1-q^{-i})^{-a_i} \ge 4^{-a_i} q^{d_in^2} \ge q^{-2a_i+d_in^2}.$$
Thus,
$$\frac{\prod_{i=1}^r \bigm|\!t_i^{G_n}\!\!\bigm|}{|G_n|} \le q^{(\frac 2{n^2}(\sum_i a_i)+r-1-\sum_i d_i)n^2}.$$
By Proposition~\ref{Combined bound},
$$\prod_i |\chi(t_i)| \le q^{(\sum_i f_{a_i,x}(\delta_i)+o(1))n^2},$$
so
$$\frac{\prod_i \bigm|\!t_i^{G_n}\!\!\bigm|}{|G_n|}\frac{\prod_i |\chi(t_i)|}{\chi(1)^{r-2}} \le q^{((2-r)x + \sum_i f_{a_i,x}(\delta_i)- \sum_i d_i+(r-1)+o(1))n^2}.$$
By Proposition~\ref{Genus 0}, under the hypothesis on $\Gamma$, we have 
\begin{align*}
(2-r)x + \sum_i f_{a_i,x}(\delta_i) - \sum_i d_i + (r-1)&\le (2-r)x + \sum_i f_{a_i,x}(\delta_i)-\sum_i \frac 1{a_i} - \sum_i \frac{a_i \delta_i^2}{a_i-1} + (r-1) \\
&< -\epsilon x  -\sum_i \frac 1{a_i} + (r-1),
\end{align*}
and the proposition follows.
\end{proof}

The positive genus variant of this result is as follows:
\begin{lem}
\label{higher non-triv char}
Let $q$ be any sufficiently large prime power.
There exists $N$ such that for all $g\ge 1$,  $r$  non-negative and positive if $g=1$, $2\le a_1\le a_2\le \cdots\le a_r$, 
such that $q$ is prime to $a_i$ for all $i$,
$n>N$, $t_i\in G_n$ satisfying $t_i^{a_i}=1$, and non-linear irreducible characters $\chi$ of $G_n$, we have
$$\prod_i \bigm|\!t_i^{G_n}\!\!\bigm| |G_n|^{2g-1}\frac{\prod_i |\chi(t_i)|}{\chi(1)^{2g+r-2}}
\le q^{n^2\bigl(2g-1+\sum_i (1-\frac 1{a_i})\bigr)}\chi(1)^{-.44}.$$
\end{lem}

\begin{proof}
The proof is exactly the same as that of Proposition~\ref{non-triv char} except that we use Lemma~\ref{Positive Genus} instead of Proposition~\ref{Genus 0}
\end{proof}

\begin{lem}
\label{zeta}
Let $H_i = \GL_{n_i}(q_i)$, where $\lim_{i\to \infty} n_i = \infty$.   Then for all $\epsilon > 0$,
$$\sum_{\chi \in \Irr(H_i)} \chi(1)^{-\epsilon} = q_i-1+o\bigl(q_i^{-\frac{\epsilon n_i}3}\bigr) $$
\end{lem}

\begin{proof}
By \cite[Theorem 1.2]{LiSh2},
$$\sum_{\chi \in \Irr(\SL_{n_i}(q_i))} \chi(1)^{-\epsilon/2} = 1+o(1),$$
so if $D_i$ denotes the minimum degree of a non-trivial character of $\SL_{n_i}(q_i)$,
$$\sum_{\chi \in \Irr(\SL_{n_i}(q_i))} \chi(1)^{-\epsilon} = 1 + o(D^{-\epsilon/2}) = 1 + o\bigl(q_i^{-\frac{\epsilon n_i}3}\bigr)$$
by \cite{LS}.  The relation which assigns to each element of $\Irr(H_i)$ all the elements in $\Irr(\SL_{n_i}(q_i))$ which are constituents of its restriction is at most $q-1$ to $1$
and non-increasing in degree.  There are $q_i-1$ linear characters for $H_i$, all mapping to the trivial character of $\SL_{n_i}(q_i))$.  Therefore,
$$1-q_i+\sum_{\chi \in \Irr(H_i)} \chi(1)^{-\epsilon} = o((q_i-1) q_i^{-\epsilon n_i/3}),$$
and the lemma follows.
\end{proof}

We can now prove Theorems \ref{Main} and \ref{Roots}.

\begin{proof}
We assume first that $g=0$, so $\Gamma$ is determined by $2\le a_1\le a_2\le \cdots\le a_r$.  
A homomorphism $\Gamma\to G_n$ is determined by the images $t_1,\ldots,t_r\in G_n$ of $z_1,\ldots,z_r\in\Gamma$, which 
satisfy $t_i^{a_i}=1$ and $\prod_i \det(t_i)=1$.
We can partition the set of homomorphisms according to the conjugacy classes $C_1,\ldots,C_r$ to which the $t_i$ belong.  
By the Frobenius formula, the total number of homomorphisms is
\begin{equation}
\label{total sum}
\sum_{(C_1,\ldots,C_r)}\frac{|C_1\times \cdots\times C_r|}{|G_n|}\sum_{\chi\in \Irr(G_n)}\frac{\chi(C_1)\cdots\chi(C_r)}{\chi(1)^{r-2}}.
\end{equation}
The determinant condition implies that each linear character in the inner sum contributes $1$, and there are a total of $q-1$
such characters.  Their total contribution is therefore
\begin{equation}
\label{linear contribution}
(q-1)\sum\frac{|C_1\times \cdots\times C_r|}{|G_n|} = (q-1)J_{q,n}(a_1,\ldots,a_r)|G_n|^{-1}
\end{equation}

By Lemma~\ref{zeta}, the contribution of all non-linear characters $\chi$ to \eqref{total sum} is $o(q^{-\epsilon n/3}q^{(1-\chi(\Gamma))n^2})$.  By Corollary~\ref{J of ai},
$$|\Hom(\Gamma,G)| = 
(q-1)q^{(r-a_1-\cdots-a_r)/24} \bigl(f_n\bigl(\frac{i\log q}{2\pi}\bigr)+o(1)\bigr)  q^{(1-\chi(\Gamma)))n^2},$$
which implies Theorem~\ref{Main} and Theorem~\ref{Roots} in the genus $0$ case.
The proof in the higher genus case is the same except that we use Lemma~\ref{higher non-triv char} instead of Proposition~\ref{non-triv char}.
\end{proof}

\begin{prop}
\label{Small centralizers}
Let $a$ and $n$ be positive integers, $a\ge 2$, and let $q$ be a prime power which is $1 \pmod{a}$.   The minimum dimension of the centralizer in $\GL_n$ of a semisimple element  $t\in \GL_n(q)$ of order dividing $a$ is $\frac{n^2}a + a\{\frac na\}\{\frac{-n}a\}$.  
If $a$ is odd, $t$ can be chosen to have determinant $1$.
If $a$ is even and $\frac na\not\in\Z$, then $t$ can be chosen to have determinant $1$ or $-1$.
If $a$ is even and $\frac na\in\Z$, then $t$ must have determinant $(-1)^{\frac na}$; if this is $-1$, there is no element in $\GL_n(q)$ whose centralizer has
dimension $\frac {n^2}a+1$, but there is an element $t'\in \SL_n(q)$ with centralizer dimension 
 $\frac{n^2}a+2$.
\end{prop}

\begin{proof}
If the multiplicities $m_1,\ldots,m_a$ of the eigenvalues $\zeta_a,\ldots,\zeta_a^a$ of a semisimple $t\in \GL_n(q)$ satisfying $t^a=1$ are written $\frac na+\varepsilon_i$, then the centralizer of $t$ has dimension
$$\sum_i m_i^2 = \frac{n^2}a + \sum_i \varepsilon_i^2.$$
As $\sum_i \varepsilon_i = 0$, either all are zero (which can only happen in the case that $a$ divides $n$, or at least one is positive and at least one is negative.  
In the latter case, if any
$\varepsilon_i\ge 1$, then by reducing this by $1$ and increasing some negative $\varepsilon_j$ by $1$, we decrease $\sum_i\varepsilon_i^2$, and likewise if some $\varepsilon_i\le -1$.
As all $\varepsilon_i$ are $\{\frac na\}$ (mod $1$), each must be $\{\frac na\}$ or $\{\frac na\}-1$, and since they sum to zero, there must be $a-a\{\frac na\}$ of the latter and $a\{\frac na\}$,
implying $\sum_i \varepsilon_i^2 = a\{\frac na\}\{\frac{-n}a\}$.

Next, we claim that as long as $\{\frac na\}\neq 0$, there exists some sequence $m_1,\ldots,m_a$ consisting of $a\{\frac na\}$ copies of $\lceil \frac na\rceil$ and $a-a\{\frac na\}$ copies of 
$\lfloor\frac na\rfloor$ such that $\prod_i \zeta_a^{im_i}$ is any desired power of $\zeta_a$.  To prove this,
it suffices to show that if $0<k<a$, the sums of $k$-element subsets $S$ of $\{0,1,\ldots,a-1\}$ represent all residue classes (mod $a$).
Indeed, if $S\neq \{a-1,a-2,\ldots,a-k\}$, there exists $s\in S$ such that $s+1\in \{0, 1,\ldots,a-1\}\setminus S$.  Thus, the set of sums of $k$-element subsets $S$ includes all integers from
$\binom k2$ to $\binom a2-\binom{a-k}a$, a total of $k(a-k)+1\ge a$ consecutive integers, which therefore represent all congruence classes (mod $a$).

Finally, assume $a$ divides $n$, so $m_1=\cdots = m_a = \frac na$ gives the minimum value $\frac{n^2}a$ of $\sum_i m_i^2$.  
Any other choice of $(m_1,\ldots,m_a)$ must have all $\varepsilon_i$ integral and at least two non-zero, so $\sum_i m_i^2 \ge \frac{n^2}a+2$.
If $m_1=\cdots=m_a$ and
$\frac na$ is even or $a$ is odd, then $\sum_i im_i$ is divisible by $a$,
so $\det(t)=1$.  If $a$ is even and $\frac na$ is odd, then $m_1=\cdots=m_a$ gives $\det(t)=-1$.
In this last case, setting $\varepsilon_1 = 1$
and $\varepsilon_{a/2+1}=-1$ and all other $\varepsilon_i=0$, we get $\sum_i im_i$ is divisible by $a$ and $\sum_i m_i^2 = \frac{n^2}a+2$.
\end{proof}

Let $E_\Gamma$ denote the set of $i$ such that $a_i$ is even.
As in the statement of Theorem~\ref{Dimension}, for each positive integer $n$, we define $\sigma_{\Gamma,n}:=-1$ if
$\frac n{a_i}\in \Z$ for all $i\in E_\Gamma$, and $\sum_{i\in E_\Gamma} \frac n{a_i}$
is odd; otherwise $\sigma_{\Gamma,n}:=1$.

\begin{prop}
\label{Optimal tuple}
Suppose $q\equiv 1\pmod{a_i}$ for all $i$.  
If $(t_1,\ldots,t_r)$ is an $r$-tuple of semisimple elements in $\GL_n(q)$ such that $t_i^{a_i}=1$ and $\prod_i \det(t_i)=1$, the minimum possible
sum of the dimensions of the centralizers of the $t_i$ in $\GL_n$ is 
$$1-\sigma_{\Gamma,n}+\sum_{i=1}^r  \Bigl( \frac{n^2}{a_i} + a_i\Bigl\{\frac n{a_i}\Bigr\}\Bigl\{-\frac n{a_i}\Bigr\}\Bigr).$$
\end{prop}

\begin{proof}
If there is at least one $a_i$ which is even and such that $\frac n{a_i}\not\in\Z$, then we can choose $t_i$ to have either determinant $1$ or $-1$
and centralizer dimension 
\begin{equation}
\label{least centralizer}
\frac{n^2}{a_i}+a_i\Bigl\{\frac n{a_i}\Bigr\}\Bigl\{-\frac n{a_i}\Bigr\}.
\end{equation}
For $j\neq i$, we can choose $t_j$ to have determinant in $\{\pm 1\}$ and centralizer $\frac{n^2}{a_j}+a_j\Bigl\{\frac n{a_j}\Bigr\}\Bigl\{-\frac n{a_j}\Bigr\}$.
Therefore, we can choose minimal centralizer dimension for all $t_i$ while imposing the condition $\prod_i \det(t_i)=1$.

If $\frac n{a_i}\in\Z$ for all $a_i$ even, and the set of $i$ such that $a_i$ is even and $\frac{n}{a_i}$ is odd has even cardinality, then we may choose
$t_i$ whose centralizer has dimension \eqref{least centralizer} for all $i$ and such that $\det(t_i) = 1$ except when $a_i$ is even and $\frac n{a_i}$ is odd.
In these cases, of which there are an odd number, $\det(t_i) = 1$, so again $\prod_i \det(t_i) = 1$.

What remains is the case $\sigma_{\Gamma,n} = -1$, and here if $t_i$ has centralizer dimension \eqref{least centralizer} for all $i$, then the product $\prod_i \det(t_i)$ is $-1$ times a product of 
elements of odd order, so it cannot be $1$.  On the other hand, if we choose one $t_i$ with $a_i$ even and $\frac n{a_i}$ odd, and choose it to have determinant $1$
and centralizer dimension $\frac {n^2}{a_i}+2$, and all other $t_j$ have minimal centralizer dimension and $\det(t_j)=\pm 1$, then the product $\prod_i \det(t_i) = 1$.
\end{proof}

We now prove Theorem~\ref{Dimension}.
\begin{proof}
By Theorem~\ref{Roots}, there exist $Q$ and $N$ such that if $q>Q$ is relatively prime to $A$ and $n>N$, then
$$\frac 12 < \frac{q^{(1-2g)n^2}|\Hom(\Gamma,\GL_n(q))|}{qJ_{q,n}(a_1,\ldots,a_r)} < \frac 32.$$

For any fixed $q$ and $n$, let $X_{q,n}$ denote the variety $\Hom(\Gamma,\GL_n)$ over the field $\F_q$.
Then, for all positive integers $m$,
$$\frac 12 < \frac{q^{(1-2g)mn^2}|X_{q,n}(\F_{q^m})|}{q^mJ_{q^m,n}(a_1,\ldots,a_r)} < \frac 32.$$
By Proposition~\ref{Optimal tuple},
\begin{align*}
\dim X_{q,n} &= \limsup_m \log_{q^m} |X_{q,n}(\F_{q^m})| \\
&= 1+(2g-1)n^2 + \limsup_m \log_{q^m} J_{q^m,n}(a_1,\ldots,a_r)\\
&=  1+(2g-1)n^2 +rn^2- 1 + \sigma_{\Gamma,n}
-\sum_{i=1}^r  \Bigl( \frac{n^2}{a_i} + a_i\Bigl\{\frac n{a_i}\Bigr\}\Bigl\{-\frac n{a_i}\Bigr\}\Bigr)\\
&=  \sigma_{\Gamma,n} + (1-\chi(\Gamma))n^2
- \sum_{i=1}^r a_i\Bigl\{\frac n{a_i}\Bigr\}\Bigl\{-\frac n{a_i}\Bigr\}.
\end{align*}
The first claim of the theorem follows in the positive characteristic case.

For characteristic zero, we consider the scheme $\Hom(\Gamma,\GL_{n,\Z})$ over $\Spec \Z$ whose $\F_p$ fiber is the $n$-dimensional representation variety of $\Gamma$ over $\F_p$.
By the constructibility of the set of dimensions of irreducible fiber components \cite[Proposition 9.5.5]{EGA}, the dimension of the generic fiber must be the same as the common dimension of any infinite set of fibers over closed points.

Finally, for the second claim of the theorem, we observe that $\{t\}\{-t\} \ge -\frac 14$ for all real $t$, so for all $i$,
$$a_i\bigl\{\frac n{a_i}\bigr\}\bigl\{-\frac n{a_i}\bigr\} \ge -\frac{a_i}4.$$
We have $\sigma_{\Gamma,n} > -\frac 12$ unless there is at least one value of $i$ for which $a_i$ is even and $\frac n{a_i}$ is integral.  For this value of $i$, $a_i\ge 2$, so 
$$\sigma_{\Gamma,n} + a_i\bigl\{\frac n{a_i}\bigr\}\bigl\{-\frac n{a_i}\bigr\} \ge -\frac 12-\frac{a_i}4,$$
which finishes the proof.
\end{proof}

\end{document}